\documentclass[final,onefignum,onetabnum]{siamart190516}

\usepackage{xcolor}
\usepackage{subcaption}
\usepackage{enumitem}
\usepackage{graphicx}
\usepackage{cite}
\usepackage{amsfonts}
\usepackage{amssymb}

\usepackage{tikz,tkz-euclide}
\usepackage[european]{circuitikz}

\usetikzlibrary{matrix,arrows,calc,fit,shapes.callouts,shapes.arrows}
\tikzset{font={\fontsize{10pt}{9}\selectfont}}

\tikzstyle{block} = [draw, rectangle, thick,minimum height=1em,minimum width=1em]
\tikzstyle{smallsum} = [draw,circle,inner sep=0mm,minimum size=3mm]
\tikzstyle{branch} = [draw,circle,inner sep=0.5mm,fill=black]
\tikzstyle{none} = [draw=none]
\tikzstyle{connector} = [->,thick]
\tikzstyle{line} = [thick]

\tikzset{%
  saturation block/.style={%
    draw, thick,
    path picture={
      \pgfpointdiff{\pgfpointanchor{path picture bounding box}{north east}}%
        {\pgfpointanchor{path picture bounding box}{south west}}
      \pgfgetlastxy\x\y
      \tikzset{x=\x*.4, y=\y*.4}
      \draw (-.9,0) -- (.9,0) (0,-.9) -- (0,.9);
      \draw (-.9,-.6) -- (-.6,-.6) -- (.6,.6) -- (.9,.6);
      \node[text width=.1cm] at (-.35,.55) {\scriptsize $P_\calU$};
    }
  }
}

\ifpdf
	\DeclareGraphicsExtensions{.eps,.pdf,.png,.jpg}
\else
	\DeclareGraphicsExtensions{.eps}
\fi

\newsiamremark{remarkx}{Remark}
\newsiamremark{examplex}{Example}
\newsiamremark{assumptionx}{Assumption}
\newenvironment{example}{\examplex}{\oprocend\endexamplex}
\newenvironment{remark}{\remarkx}{\oprocend\endremarkx}
\newenvironment{assumption}{\assumptionx}{\oprocend\endremarkx}
\newcommand\oprocendsymbol{\hbox{\small $\blacksquare$}}
\newcommand\oprocend{\relax\ifmmode\else\unskip\hfill\fi\oprocendsymbol}

\graphicspath{{./Figs/}{./Simulations}}
\numberwithin{theorem}{section}
\numberwithin{equation}{section}

\newcommand{\bbR}{\mathbb{R}}
\newcommand{\calX}{\mathcal{X}}
\newcommand{\bbB}{\mathbb{B}}
\newcommand{\calL}{\mathcal{L}}
\newcommand{\bbS}{\mathbb{S}}
\newcommand{\calU}{\mathcal{U}}
\newcommand{\bbI}{\mathbb{I}}
\newcommand{\calZ}{\mathcal{Z}}
\newcommand{\tF}{\hat{F}}
\newcommand{\calV}{\mathcal{V}}
\newcommand{\calA}{\mathcal{A}}
\newcommand{\calC}{\mathcal{C}}
\newcommand{\calK}{\mathcal{K}}
\newcommand{\calKL}{\mathcal{KL}}
\newcommand{\calB}{\mathcal{B}}
\renewcommand{\o}[1]{\overline{#1}}
\DeclareMathOperator{\dom}{dom}
\DeclareMathOperator{\co}{co}
\DeclareMathOperator{\cl}{cl}
\DeclareMathOperator{\cocl}{\overline{co}}
\DeclareMathOperator{\gph}{gph}
\DeclareMathOperator{\interior}{int}

\newcommand{\pds}{PDS }
\newcommand{\pdss}{PDS }

\newenvironment{smallbmatrix}
  {\left[\begin{smallmatrix}}
  {\end{smallmatrix}\right]}

\usepackage{pgfplots}
\pgfplotsset{compat=newest}
\usetikzlibrary{plotmarks}
\usetikzlibrary{arrows.meta}
\usepgfplotslibrary{patchplots}
\tikzset{font={\fontsize{8pt}{9}\selectfont}}
\usepackage{grffile}
\usepackage{amsmath}

\ifpdf
\hypersetup{
	pdftitle={Anti-Windup Implementation of Projected Dynamics},
	pdfauthor={A. Hauswirth, F. D\"orfler, and A. R. Teel}
}
\fi

\headers{Anti-Windup Implementation of Projected Dynamics}{A. Hauswirth, F.D\"orfler, and A. Teel}

\title{Anti-Windup Approximations of Oblique Projected Dynamics for Feedback-based Optimization\thanks{Submitted to the editors DATE.
\funding{This work was supported by ETH Zurich funds, SNF AP Energy grant \#160573 and mobility grant \#160573/2, NSF grant ECCS-1508757 and AFOSR grant FA9550-18-1-0246.}}}

\author{Adrian Hauswirth\thanks{Department of Information Technology and Electrical Engineering, ETH Z\"urich,
                Zurich, Switzerland (\email{hadrian@ethz.ch},\email{dorfler@ethz.ch}).}
\and Florian D\"orfler\footnotemark[2]
\and Andrew R. Teel\thanks{Department of Electrical and Computer Engineering, University of California, Santa Barbara, CA (\email{teel@ece.ucsb.edu}).}
}

\begin{document}

\maketitle

\begin{abstract}
	In this paper we study how high-gain anti-windup schemes can be used to implement projected dynamical systems in control loops that are subject to saturation on a (possibly unknown) set of admissible inputs. This insight is especially useful for the design of autonomous optimization schemes that realize a closed-loop behavior which approximates a particular optimization algorithm (e.g., projected gradient or Newton descent) while requiring only limited model information.
	In our analysis we show that a saturated integral controller, augmented with an anti-windup scheme, gives rise to a perturbed projected dynamical system. This insight allows us to show uniform convergence and robust practical stability as the anti-windup gain goes to infinity. Moreover, for a special case encountered in autonomous optimization we show robust convergence, i.e., convergence to an optimal steady-state for finite gains. Apart from being particularly suited for online optimization of large-scale systems, such as power grids, these results are potentially useful for other control and optimization applications as they shed a new light on both anti-windup control and projected gradient systems.
\end{abstract}



\section{Introduction}
In recent years, the design of feedback controllers based on optimization algorithms has garnered significant interest as a new approach to real-time optimization of large-scale systems such as power grids~\cite{dallaneseOptimalPowerFlow2018,tangRealTimeOptimalPower2017,hauswirthOnlineOptimizationClosed2017,molzahnSurveyDistributedOptimization2017} and communication networks~\cite{lowInternetCongestionControl2002,kellyRatecontrolcommunication1998}.
The goal of \emph{autonomous} (or \emph{feedback-based) optimization} is to implement feedback systems that robustly solve nonlinear optimization problems in closed loop with a physical system, often without requiring explicit knowledge of the problem parameters, because the physical plant itself enforces certain constraints.

In this paper, we investigate a new approach to enforce constraints by exploiting physical saturation. More precisely, we study how anti-windup control, which is ubiquitous in feedback control to mitigate integrator windup, can be used to implement projected dynamical systems (PDS) which are at the basis of continuous-time algorithms for constrained optimization.
In particular, \pdss form a class of discontinuous dynamical systems that encompasses projected gradient flow~\cite{hauswirthProjectedGradientDescent2016}, projected Newton flow~\cite{hauswirthProjectedDynamicalSystems2018}, subgradient flow~\cite{cortesDiscontinuousdynamicalsystems2008} and projected saddle-flows~\cite{goebelStabilityRobustnessSaddlepoint2017,cherukuriRoleConvexitySaddlePoint2017}. More generally, \pdss arise in many contexts that include unilateral constraints, such as variational inequalities~\cite{nagurneyProjectedDynamicalSystems1996,facchineiFinitedimensionalvariationalinequalities2003}, evolutionary games~\cite{lahkarProjectionDynamicGeometry2008}, and complementarity systems~\cite{brogliatoEquivalenceComplementaritySystems2006,aubinDifferentialInclusionsSetValued1984}.

The main contribution of this paper is to establish a rigorous connection between \pdss and anti-windup controllers and to generalize~\cite{hauswirthImplementationProjectedDynamical2020}. From an abstract point of view, we consider a class of parametrized dynamical systems, termed \emph{anti-windup approximations} (AWA), and we show uniform convergence of trajectories to the solution of a \pds as the \emph{anti-windup gain} tends to infinity. Moreover, we establish semiglobal practical robustness of \pdss with respect to anti-windup approximations. For the special case of strongly monotone vector fields we further show robust asymptotic stability for finite gains.

Compared to~\cite{hauswirthImplementationProjectedDynamical2020} we make the following generalizations:
\begin{enumerate}[leftmargin=*, label=\emph{\roman*})]
	\item We do not require the feasible domain to be convex. Instead, we work with (non-convex) prox-regular sets and show, by means of a counter-example, that prox-regularity cannot, in general, be relaxed further.

	\item We consider \emph{oblique} \pds that provide an additional degree in the form of a (Riemannian) metric that allows us to capture wider variety of dynamics (such as projected Newton flows) and that is required for coordinate-free formulations.

	\item We require solutions to be neither unique nor complete. In particular, our results allow for solutions with finite escape time. Although these results may not be of practical relevance, they illustrate the necessity of our assumptions.

	\item We establish requirements for the convergence of anti-windup approximations of monotone dynamics on convex domains, thus providing a (partial) solution to a previously open problem formulated in~\cite{hauswirthImplementationProjectedDynamical2020}.
\end{enumerate}

Finally, in a largely self-contained section, we illustrate the possibilities of the proposed anti-windup approximations of \pdss and the applicability of our theoretical results in the context of autonomous optimization~\cite{tangRunningPrimalDualGradient2018,colombinorobustnessguaranteesfeedbackbased2019,hauswirthProjectedGradientDescent2016,lawrenceOptimalSteadyStateControl2018,colombinoOnlineOptimizationFeedback2019}.

The rest of the paper is organized as follows: In \cref{sec:prelim} we fix the notation and recall relevant notions from variational analysis and dynamical systems. In \cref{sec:tech} we define our problem and establish some technical lemmas. In \cref{sec:unif,sec:rob} we present our first two main results (\cref{thm:loc,thm:rob_stab}) on uniform convergence and semiglobal practical robust stability. In \cref{sec:rob_conv} we provide a stronger stability guarantee (\cref{thm:monot}) for the special case of monotone vector fields.
In~\cref{sec:aw_appl}, we illustrate the consequences of our results in the context of feedback-based optimization. For this, we consider four different optimization dynamics (three gradient-based and one saddle-point flow) and discuss their convergence behavior observed in simulations. In \cref{sec:conc} we summarize our results and discuss open problems.

\section{Preliminaries}\label{sec:prelim}

\subsection{Notation}
We consider $\bbR^n$ with the usual inner product $\left\langle \cdot, \cdot \right\rangle$ and 2-norm $\| \cdot \|$. We use $\bbR^n_{\geq 0}$ for the non-negative orthant. The closed (open) unit ball of appropriate dimension is denoted by $\bbB$ ($\interior \bbB$). For a sequence $\{K_n \}$, $K_n \rightarrow 0^+$ implies that that $K_n > 0$ for all $n$ and ${\lim}_{n \rightarrow \infty} K_n = 0$.
For a map $F: \bbR^n \rightarrow \bbR^m$, differentiable at $x \in \bbR^n$, $\nabla F(x) \in \bbR^{m \times n}$ denotes the \emph{Jacobian of $F$ at $x$}.
Given a set $\calC \subset \bbR^n$, its \emph{closure}, \emph{boundary}, and \emph{(closed) convex hull} are denoted by $\cl{\calC}$, $\partial \calC$, and $\co{\calC}$ ($\cocl \calC$), respectively. If $\calC$ is non-empty, we write $\| \calC \| := \sup_{v \in \calC}  \| v \|$.
The \emph{distance to $\calC$} is defined as $d_\calC(x) := \inf_{\tilde{x} \in \calC} \| x - \tilde{x} \|$, and the \emph{projection} $P_\calC : \bbR^n \rightrightarrows \calC$ is given by $P_\calC(x) := \{ \tilde{x} \in \calC \, | \, \| x - \tilde{x} \| = d_{\calC}(x) \}$.
The \emph{domain} of a set-valued map $H: \bbR^n \rightrightarrows \bbR^m$ is defined as $\dom H := \{ x \, | \, H(x) \neq \emptyset \}$. We use the standard definitions of \emph{outer semicontinuity (osc)}, \emph{local boundedness}, \emph{graphical convergence}, etc. from set-valued analysis. In particular, unless noted otherwise, we follow the definitions and notation of~\cite[Chap.~5]{goebelHybridDynamicalSystems2012}.
The identity matrix (of appropriate size) is denoted by $\bbI$.
Given a square symmetric matrix $G \in \bbR^{n \times n}$, $\lambda^{\max}_G$ and $\lambda^{\min}_G$ denote its maximum and minimum eigenvalue, respectively.
The set of symmetric, positive definite matrices of size $n$ is denoted by $\bbS_+^n$.
A \emph{metric} on $\calC \subset \bbR^n$ is a map $G: \calC \rightarrow \bbS^n_+$. A metric $G$ induces an \emph{inner product} $\left\langle u, v \right\rangle_{G(x)} := u^T G(x) v$ and an associated \emph{2-norm} $\| u \|_{G(x)} := ( \left\langle u, u \right\rangle_{G(x)})^{1/2}$ for all $x \in \calC$ and all $u, v \in \bbR^n$. A metric is (\emph{Lipschitz}) \emph{continuous} if it is (Lipschitz) continuous as a map $G: \calC \rightarrow \bbS^n_+$ with respect to the $\lambda^{\max}$-norm on $\bbS^n_+$.

\subsection{Variational Geometry}

We use the following, slightly simplified\footnote{To allow for a more concise presentation, we limit ourselves to closed, Clarke regular subsets of $\bbR^n$ which allow for an unambiguous definition of tangent and normal cones.}, notions of variational geometry. For a comprehensive treatment the reader is referred to~\cite{rockafellarVariationalAnalysis2009}.

Given a closed set $\calC \subset \bbR^n$ and $x \in \calC$, a vector $v$ is a \emph{tangent vector to $\calC$ at $x$} if there exist sequences $x_k \rightarrow x$ with $x_k \in \calC$ for all $k$ and $\delta_k \rightarrow 0^+$ such that $\frac{x_k - x}{\delta_k}\rightarrow v$. The set of all tangent vectors at $x$ is called the \emph{tangent cone at $x$} and denoted by $T_x \calC$. If the set-valued map $x \mapsto T_x \calC$ is inner semicontinuous then $\calC$ is \emph{Clarke regular} (or \emph{tangentially regular})~\cite[Cor.~6.29]{rockafellarVariationalAnalysis2009}. If $\calC$ is Clarke regular then $T_x\calC$ is closed convex and the \emph{(Euclidean) normal cone at $x$} is defined as the polar cone of $T_x\calC$, i.e., $N_x \calC := \{ \eta \, | \, \forall v \in T_x \calC: \, \left\langle \eta, v\right\rangle \leq 0 \}$~\cite[Cor.~6.30]{rockafellarVariationalAnalysis2009}. Further, the map $x \mapsto N_x \calC$ is osc~\cite[Prop.~6.6]{rockafellarVariationalAnalysis2009}. We follow the convention that $T_x \calC = N_x \calC = \emptyset$ for all $x \notin \calC$.

We will mostly work with the special class of \emph{prox-regular sets}. Given a Clarke regular set $\calC$ and $\alpha > 0$, a normal vector $\eta \in N_x \calC$ is \emph{$\alpha$-proximal} if $\left\langle \eta, y - x \right\rangle \leq \alpha \| \eta \| \| y - x \|^2$ for all $y \in \calC$. The set $\calC$ is \emph{$\alpha$-prox-regular at $x$} if all normal vectors at $x$ are $\alpha$-proximal. In other words, the normal cone coincides with the cone of $\alpha$-proximal normals.
A set is \emph{$\alpha$-prox-regular} if it is $\alpha$-prox-regular at all $x \in \calC$ and \emph{prox-regular} if it is $\alpha$-prox-regular for some $\alpha > 0$.
A key property of prox-regular sets is that the projection on $\calC$ is locally well-defined~\cite[Thm.~2.2 \& Prop.~2.3]{adlyPreservationProxRegularitySets2016}:

\begin{proposition}\label{prop:prox}
	If $\calC \subset \bbR^n$ is $\alpha$-prox-regular, then, for every $x \in \calC + \frac{1}{2\alpha} \interior \bbB$, the set $P_\calC(x)$ is a singleton and $d^2_\calC$ is differentiable at $x$ with $\nabla (d^2_\calC(x)) = 2(x - P_\calC(x))$.
	Further, $P_\calC(\o{x} + v) = \o{x}$ holds for every $\o{x} \in \calC$ and all $v \in N_{\o{x}} \calC \cap \frac{1}{2\alpha} \interior \bbB$.
\end{proposition}

For example, every closed convex set is Clarke regular as well as $\alpha$-prox-regular for all $\alpha > 0$. Further, every set of the form $\calC = \{ x \, | \, h(x) \leq 0 \}$, where $h: \bbR^n \rightarrow \bbR^m$ is differentiable, is Clarke regular if constraint qualifications hold~\cite[Thm.~6.14]{rockafellarVariationalAnalysis2009}. If, in addition, $h$ has a globally Lipschitz derivative, then $\calC$ is prox-regular~\cite[Ex.~7.7]{hauswirthProjectedDynamicalSystems2018}.

\subsection{Dynamical Systems \& Stability}

Given a closed set $\calC \subset \bbR^n$ and a set-valued map $H: \bbR^n \rightrightarrows \bbR^n$, we say that $x:[0, T ] \rightarrow \calC$ for some $T>0$ is a \emph{(Carath\'eodory) solution} of the (constrained) differential inclusion
\begin{align}\label{eq:gen_inclusion}
	\dot x \in H(x) \, , \qquad x \in \calC \,
\end{align}
if $x$ is absolutely continuous, and $x(t) \in \calC$ and $\dot x(t) \in H(x(t))$ hold for almost all $t \in [0,T]$. A map $x : [ 0, \infty) \rightarrow \calC$ is a \emph{complete} solution, if its restriction to any compact subinterval $[0, T]$ is a solution of~\eqref{eq:gen_inclusion}.

\begin{definition} An inclusion~\eqref{eq:gen_inclusion} is \emph{well-posed} if
	$\calC$ is closed, $H$ is osc and locally bounded relative to $\calC$, and $H(x)$ is non-empty and convex for all $x \in \calC$.
\end{definition}

Standard results (e.g.,~\cite[Lem.~5.26]{goebelHybridDynamicalSystems2012}) guarantee that~\eqref{eq:gen_inclusion} admits a solution for every initial condition $x(0) \in \calC$ if it is well-posed and $H(x) \cap T_x \calC \neq \emptyset$ for all $x \in \calC$.

For convenience, we introduce the following notion of truncated solution:
\begin{definition}
	Consider~\eqref{eq:gen_inclusion} with $\calC = \bbR^n$. Given $T, \epsilon > 0$ and $x_0 \in \bbR^n$, a solution $x:[0, T'] \rightarrow \bbR^n$ of~\eqref{eq:gen_inclusion} with initial condition $x(0) = x_0$ is \emph{$(T, \epsilon)$-truncated} if $x(t) \in x_0 + \epsilon \bbB$ for all $t \in [0, T']$ and either $T' = T$ or $\| x(T') - x_0 \| = \epsilon$ holds.
\end{definition}

Recall that on a compact domain, solutions of an (unconstrained) inclusion can always be extended up to the boundary of the domain~\cite[\S7, Thm.~2]{filippovDifferentialEquationsDiscontinuous1988}:

\begin{theorem}\label{thm:filippov}
	Let~\eqref{eq:gen_inclusion} be well-posed with $\calC = \bbR^n$ and let $\calA \subset \bbR^n$ be compact. Then, every solution $x: [0, T] \rightarrow \bbR^n$ with $x(0) \in \calA$ can be extended up to the boundary of $\calA$, i.e., there is a solution for every  $T > 0$ or there exists $T$ such that $x(T) \in \partial \calA$.
\end{theorem}

Therefore, by considering an augmented inclusion with $\hat{H}(x) := (H(x), 1)$, initial condition $\hat{x}(0) := (x(0), 0)$, $\hat{\calC} = \bbR^n \times \bbR$, and $\hat{\calA} = \calA \times [0, T]$, \cref{thm:filippov} guarantees the existence of truncated solutions for every $T$ and every $\epsilon$:
\begin{corollary}
	Let \eqref{eq:gen_inclusion} be well-posed with $\calC = \bbR^n$. Then, for every $T, \epsilon > 0$ and every $x(0) \in \bbR^n$ there exists a $(T, \epsilon)$-truncated solution to~\eqref{eq:gen_inclusion}.
\end{corollary}

Hence, truncated solutions are convenient if finite escape times cannot be precluded, since their graph is always a compact subset of $[0, T] \times (x(0) + \epsilon \bbB)$.

We also require the notion of $\sigma$-perturbation of an inclusion~\cite[Def.~6.27]{goebelHybridDynamicalSystems2012}:

\begin{definition}
	Given $\sigma > 0$, the \emph{$\sigma$-perturbation of~\eqref{eq:gen_inclusion}} is given by
	\begin{align*}
		\dot x \in H_\sigma (x) \qquad x \in \calC_\sigma
	\end{align*}
	where $\calC_\sigma := \calC + \sigma \bbB$ and $H_\sigma(x) := \cocl H(( x + \sigma \bbB) \cap \calC) + \sigma \bbB$ for all $x \in \calC_\sigma$.
\end{definition}

Note in particular that, for $\sigma' \geq \sigma$,we have $\calC_{\sigma} \subset \calC_{\sigma'}$, $H_{\sigma}(x) \subset H_{\sigma'}(x)$ for all $x \in \calC_\sigma$, and every solution of the $\sigma$-perturbation is a solution of the $\sigma'$-perturbation.

Next, recall that $\omega: \bbR_{\geq 0} \rightarrow \bbR_{\geq 0}$ is a \emph{$\calK_\infty$-function} (denoted by $\omega \in \calK_\infty$) if $\omega$ is continuous, strictly increasing, unbounded, and it holds that $\omega(0) = 0$.
We require the following lemma about $\calK_\infty$-functions:

\begin{lemma}\cite[Cor.~10]{sontagCommentsIntegralVariants1998}\label{lem:kinf}
	For every $\omega \in \calK_\infty$, there exist $\sigma_1, \sigma_2 \in \calK_\infty$ such that $\omega(rs) \leq \sigma_1(r)\sigma_2(s)$ for all $r, s \geq 0$.
\end{lemma}

A function $\beta:\bbR_{\geq 0} \times \bbR_{\geq 0} \rightarrow \bbR_{\geq 0}$ is a \emph{$\calKL$-function} (denoted by $\beta \in \calKL$) if it is non-decreasing in its first argument, non-increasing in its second argument, $\lim_{r \rightarrow 0^+} \beta(r,s) = 0$ for each $s \in \bbR_{\geq 0}$, and $\lim_{s \rightarrow \infty} \beta(r,s) = 0 $ for each $t \in \bbR_{\geq 0}$.

A closed set $\calA \subset \bbR^n$ is \emph{uniformly globally (pre-)asymptotically stable} for~\eqref{eq:gen_inclusion} if there exists $\beta \in \calKL$ such that for every solution $x: [0, T] \rightarrow \calC$ of \eqref{eq:gen_inclusion} it holds that
\begin{align*}
	d_{\calA}(x(t)) \leq \beta( d_{\calA}(x(0)), t) \qquad \forall t \in [0, T] \, .
\end{align*}

\begin{remark}\label{rem:pre_stab}
	The term ``pre-asymptotic'' refers to the fact that solutions of~\eqref{eq:gen_inclusion} need not be complete for the above definition of stability to apply~\cite[Def 3.6 \& Thm 3.40]{goebelHybridDynamicalSystems2012}. However, if~\eqref{eq:gen_inclusion} is well-posed and $\calA$ is compact it follows that, for any initial condition $x(0) \in \calC$, the (compact) set $\{x \, | \, \beta(d_\calA(x), 0) \leq \beta(d_\calA(x(0)), 0) \}$ is invariant, thus implicitly guaranteeing the existence of a complete solution.
\end{remark}

\subsection{Oblique Projected Dynamical Systems}
PDS are continuous-time dynamical systems that are constrained to a set by projection of the vector field at the boundary of the domain.
Compared to traditional definitions~\cite{brogliatoEquivalenceComplementaritySystems2006,cornetExistenceslowsolutions1983,aubinDifferentialInclusionsSetValued1984,nagurneyProjectedDynamicalSystems1996}, we incorporate the possibility of \emph{oblique} projection directions by means of a variable metric~\cite{hauswirthProjectedDynamicalSystems2018}.
Namely, we consider \pdss as defined by the differential equation of the form
\begin{align}\label{eq:pds_ivp}
	\dot x = \Pi^G_\calC [ f(x)] (x) \, \qquad x \in \calC \, ,
\end{align}
where $\calC \subset \bbR^n$ is a Clarke regular set, $G:\calC \rightarrow \bbS_+^n$ is a metric on $\calC$, and $f: \bbR^n \rightarrow \bbR^n$ is a vector field. Given $x \in \calC$ and $w \in \bbR^n$, the operator $\Pi^G_\calC$ projects $w$ onto the tangent cone of $\calC$ at $x$ with respect to the metric $G$, i.e.,
\begin{align*}
	\Pi^G_\calC [ w ] (x) := \underset{v \in T_x \calC}{\arg \min} \| v -w \|_{G(x)} \, .
\end{align*}
Note that if $x \in \calC$, then $\Pi^G_\calC[w](x)$ is single-valued since $\calC$ is assumed to be Clarke regular which implies that $T_x \calC$ is closed convex. If $x \notin \calC$, we have $\Pi^G_\calC [ w ] (x) = \emptyset$ and therefore $\dom \Pi^G_\calC [ w ] = \calC$ for all $w \in \bbR^n$. If $f$ is a vector field, we abuse notation and write $\Pi^G_\calC [ f ] (x) := \Pi^G_\calC [ f(x) ] (x)$ for brevity.

Given a metric $G$, we define the \emph{normal cone of $\calC$ at $x$ with respect to $G$} as $N_x^G \calC := \{ \eta \, | \, \forall v \in T_x \calC: \, \left\langle \eta, v \right\rangle_{G(x)} \leq 0\}$. Note in particular that we have
\begin{align}\label{eq:norm_cone_equiv}
	\eta \in N_x \calC \quad \Longleftrightarrow \quad G^{-1}(x) \eta \in N^G_x \calC \, .
\end{align}

As a consequence of Moreau's Theorem~\cite[Thm.~3.2.5]{hiriart-urrutyFundamentalsConvexAnalysis2012} the operator $\Pi_\calC^G$ has the following crucial properties (see also~\cite{brogliatoEquivalenceComplementaritySystems2006,cornetExistenceslowsolutions1983,aubinDifferentialInclusionsSetValued1984}):

\begin{lemma}\cite[Lem.~4.5]{hauswirthProjectedDynamicalSystems2018}\label{lem:pds_decomp}
	If $\calC$ is Clarke regular then, for every $x \in \calC$, there exists a unique $\eta \in N^G_x \calC$ such that $\Pi^G_\calC[f](x) = f(x) - \eta$. Furthermore, $\Pi^G_\calC[f](x) = f(x) - \eta$ holds if and only if $\eta \in N^G_x \calC$ and $\left\langle f(x) - \eta, \eta \right\rangle_{G(x)} = 0$. Using Cauchy-Schwarz, it also holds that $\| \eta \|_{G(x)} \leq \| f(x) \|_{G(x)}$.
\end{lemma}

Existence and uniqueness results for~\eqref{eq:pds_ivp} without a variable metric can be found in~\cite{aubinDifferentialInclusionsSetValued1984,nagurneyProjectedDynamicalSystems1996,cornetExistenceslowsolutions1983} and others.
For the case with a variable metric with bounded condition number, the following statement is a condensation of results in~\cite{hauswirthProjectedDynamicalSystems2018}:

\begin{theorem}\label{thm:pds_exist}
	Consider~\eqref{eq:pds_ivp} and let $\calC$ be Clarke regular, and $f$ and $G$ be continuous. Then,~\eqref{eq:pds_ivp} admits a solution for every initial condition $x(0) \in \calC$.

	If, in addition, there exists $\kappa > 0$ such that
	$\sup_{x \in \calC} \lambda^{\max}_{G(x)}/ \lambda^{\min}_{G(x)} \leq \kappa$ and $f$ is globally Lipschitz,
	then~\eqref{ex:pds_exist} admits a complete solution for every $x(0) \in \calC$.

	If $\calC$ is prox-regular, and if $f$ and $G$ are (locally) Lipschitz, then~\eqref{eq:pds_ivp} admits a unique solution for every initial condition $x(0) \in \calC$.\footnote{
	A solution $x:[0, T] \rightarrow \calC$ of~\eqref{eq:pds_ivp} is unique if for every other solution $x':[0, T'] \rightarrow \calC$ with the same initial condition it holds that $x(t) = x'(t)$ for all $t \in [0, \min\{T, T'\}]$.}
\end{theorem}

It is known that the solutions to~\eqref{eq:pds_ivp} are equivalent to the solutions of $\dot x \in f(x) - N^G_x \calC$ (for $G = \bbI$ see~\cite{aubinViabilityTheory1991,cornetExistenceslowsolutions1983}; for general $G$ see~\cite[Cor.~6.3]{hauswirthProjectedDynamicalSystems2018}). In light of \cref{lem:pds_decomp}, we can show (next) that solutions of~\eqref{eq:pds_ivp} are equivalent to solutions of
\begin{align}\label{eq:ivp_alt}
	\dot x \in F(x) := f(x) - N^G_x \calC  \cap \gamma \bbB \qquad x \in \calC \,,
\end{align}
where $\gamma \geq \sup_{x \in \calC} \| f(x) \|_{G(x)}$ (assuming $\sup_{x \in \calC} \| f(x) \|_{G(x)} < \infty$).
The advantage of this latter inclusion is that the mapping $F$ is bounded.

\begin{proposition}\label{prop:pds_alt_form}
	If $\calC$ is Clarke regular and $x \mapsto \| f(x) \|_{G(x)}$ is bounded, then, $x: [0, T] \rightarrow \calC$ with $T>0$ is a solution of~\eqref{eq:pds_ivp} if and only if it is a solution of~\eqref{eq:ivp_alt}.
\end{proposition}

\begin{proof}
	Let $x:[0, T] \rightarrow \calC$ be a solution of~\eqref{eq:pds_ivp}. Then, $\Pi^G_\calC[f](x(t)) = f(x(t)) - \eta(t)$ for some $\eta(t) \in N^G_{x(t)}\calC$ satisfying $\| \eta(t) \|_{G(x(t))} \leq \| f(x(t)) \|_{G(x(t))} \leq \gamma$ by \cref{lem:pds_decomp} and therefore $\eta(t) \in N^G_{x(t)} \calC \cap \gamma \bbB$.
	Conversely, assume that $x$ solves~\eqref{eq:ivp_alt}. Whenever $\dot{x}(t)$ exists, it holds that $\dot{x}(t) \in T_{x(t)} \calC \cap - T_{x(t)} \calC$ \cite[eq. 2.6]{cornetExistenceslowsolutions1983} and $\dot{x}(t) = f(x(t))- \eta(t)$ for some $\eta(t) \in N^G_{x(t)} \calC \cap \gamma \bbB$. Thus, we have
	\begin{align*}
		\left\langle f(x(t)) - \eta(t), \eta(t) \right\rangle_{G(x(t))} \leq 0
		\qquad \text{and} \qquad
		\left\langle -f(x(t)) + \eta(t), \eta(t) \right\rangle_{G(x(t))}\leq 0 \, ,
	\end{align*}
	and therefore $\left\langle f(x(t)) - \eta(t), \eta(t) \right\rangle_{G(x(t))} = 0$
	which, in turn, implies that $f(x(t)) - \eta(t) = \Pi_\calC^G[f](x(t))$ by \cref{lem:pds_decomp}.
\end{proof}

\begin{lemma}
	If $f$ and $G$ are continuous and $\calC$ Clarke regular,~\eqref{eq:ivp_alt} is well-posed.
\end{lemma}

\begin{proof}
	Non-emptiness and convexity of $F(x)$ are immediate because $N^G_x \calC \cap \gamma \bbB$ is non-empty (in particular, $0 \in N^G_x \calC$) and convex for all $x \in \calC$ (and $f$ is single-valued). For outer semicontinuity recall that for a Clarke regular $\calC$ and continuous $G$ the mapping $x \mapsto N^G_x \calC$ is osc~\cite[Lem.~A.6]{hauswirthProjectedDynamicalSystems2018}. It then follows that the truncation $ N^G_x \calC \cap \gamma \bbB$ is osc and locally bounded \cite[p.161]{rockafellarVariationalAnalysis2009}. Finally, since $f$ is continuous and single-valued, $x \mapsto f(x) - N^G_x \calC \cap \gamma \bbB$ is osc and locally bounded.
\end{proof}

\section{Problem Formulation \& Technical Results}\label{sec:tech}

Throughout the paper, we consider the system given by the (unconstrained) inclusion
\begin{align}\label{eq:sys_aw_gen}
	\dot z \in F_K(z) := f(z, P_\calZ(z)) - \tfrac{1}{K} G^{-1}(P_\calZ(z)) \big ( z - P_\calZ(z) \big) \, ,
\end{align}
where $\calZ \subset \bbR^n$ is a closed set, $f: \bbR^n \times \calZ \rightarrow \bbR^n$ is a continuous vector field, $G: \calZ \rightarrow \bbS^n_+$ is a continuous metric, and $K > 0$ is a constant parameter. Because $P_\calZ$ is in general not single-valued (unless $\calZ$ is convex),~\eqref{eq:sys_aw_gen} has to be treated as a differential inclusion.

Systems of the form~\eqref{eq:sys_aw_gen} arise in the context of anti-windup control for feedback loops with integral controllers, as will be discussed in \cref{sec:aw_appl}. Hence, we will refer to~\eqref{eq:sys_aw_gen} as an \emph{anti-windup approximation} (AWA).

We study the behavior of solutions of~\eqref{eq:sys_aw_gen} as $K \rightarrow 0^+$ and show that, under appropriate assumption on $\calZ, f$, $G$, and for an initial condition $z(0) \in \calZ$, these solutions \emph{converge uniformly} to solutions of the projected dynamical system
\begin{align}\label{eq:aw_pds}
	\dot z = \Pi^G_\calZ [\hat{f}](z) \, , \qquad z \in \calZ \, ,
\end{align}
where we use $\hat{f}(z) := f(z, P_\calZ(z))$. Further, we show that a compact globally asymptotically set for~\eqref{eq:aw_pds} is \emph{semiglobally practically asymptotically stable} for~\eqref{eq:sys_aw_gen} in~$K$. Namely, if $\calA$ is compact and asymptotically stable  compact for the PDS~\ref{eq:aw_pds}, then for any compact set of initial conditions $\calB$ and any $\zeta > 0$, there exists $K > 0$ such that all trajectories of the AWA~\eqref{eq:sys_aw_gen} starting in $\calB$ converge to a subset of $\calA + \zeta \bbB$.

The key idea for studying~\eqref{eq:sys_aw_gen} is to exploit $\alpha$-prox-regularity of $\calZ$ which, according to \cref{prop:prox} guarantees, that $P_\calZ(z)$ is single-valued for all
\begin{align*}
	z \in \calZ^\circ_\alpha := \calZ + \tfrac{1}{2\alpha} \interior \bbB \, .
\end{align*}
Hence, on $\calZ^\circ_\alpha$,~\eqref{eq:sys_aw_gen} reduces to an ODE. Further, under appropriate conditions on the problem parameters and for small enough $K$, trajectories starting in $\calZ$ remain in $\calZ^\circ_\alpha$. This insight will be rigorously established in \cref{subsec:exist}. In \cref{subsec:perturb} we then show that the AWA~\eqref{eq:sys_aw_gen} corresponds to a $\sigma$-perturbation of the PDS~\eqref{eq:aw_pds} as a function of $K$. We then apply standard results from~\cite{goebelHybridDynamicalSystems2012} to establish uniform convergence and semiglobal practical asymptotic stability in \cref{sec:unif,sec:rob}, respectively.

\subsection{Existence, Local Uniform Boundedness, and Equicontinuity}\label{subsec:exist}

As a first step in studying~\eqref{eq:sys_aw_gen}, we prove the following lemma for future reference:

\begin{lemma}\label{lem:vec_field}
	Let $\calZ \subset \bbR^n$ be closed and $f: \bbR^n \times \calZ$ be continuous. Then, $z \mapsto \hat{f}(z) := f(z, P_\calZ(z))$ is locally bounded and osc. Furthermore, if $\calZ$ is $\alpha$-prox-regular for $\alpha > 0$, then $\hat{f}$ is single-valued and continuous for all $z \in \calZ^\circ_\alpha$.
\end{lemma}

\begin{proof}
	The projection $P_\calZ: \bbR^n \rightrightarrows \calZ$ is osc and locally bounded~\cite[Ex.~5.23]{rockafellarVariationalAnalysis2009}, and $P_\calZ(z)$ is non-empty and closed for all $z \in \bbR^n$ (since $\calZ$ is closed). By continuity of $f$ it follows that $\hat{f}$ is osc and locally bounded, since both properties are preserved under addition and composition~\cite[Prop.~5.51~\&~5.52]{rockafellarVariationalAnalysis2009}. Using \cref{prop:prox} it follows that $\hat{f}$ is single-valued (hence continuous) for $z \in \calZ_\alpha^\circ$.
\end{proof}

\cref{lem:vec_field} and \cref{prop:prox} imply that, on $\calZ^\circ_\alpha$, $F_K$ is single-valued and continuous. Consequently, standard results for ODEs guarantee that~\eqref{eq:sys_aw_gen} admits a (local) solution for every initial condition $z(0) \in \calZ^\circ_\alpha$. However, outside of $\calZ^\circ_\alpha$,~\eqref{eq:sys_aw_gen} is a differential inclusion for which the existence of solutions is not immediately guaranteed.
Nevertheless, one can establish the existence of so-called~\emph{Krasovskii solutions}~\cite{goebelHybridDynamicalSystems2012}.

For the main result of this section we consider the following (local) setup:

\begin{assumption}\label{ass:local}
	Consider~\eqref{eq:sys_aw_gen} and $z_0 \in \calZ$. Let $M, \nu, \mu, \alpha, \epsilon > 0$ be such that
	\begin{align}\label{eq:lin_growth}
		\| f(z, P_\calZ(z)) \| \leq M \qquad \text{and} \qquad
		\mu \bbI \preceq G^{-1}(P_\calZ(z)) \preceq \nu \bbI
	\end{align}
	hold for all $z \in (z_0 + \epsilon \bbB) \cap \calZ^\circ_\alpha$, and $\calZ$ is $\alpha$-prox-regular at every $z \in (z_0 + \epsilon \bbB) \cap Z$.
\end{assumption}

Parameters $M, \mu, \nu, \epsilon$ that satisfy~\eqref{eq:lin_growth} can always be found for any $z_0 \in \calZ$ since $z \mapsto f(z, P_\calZ(z))$ is locally bounded by \cref{lem:vec_field}, $G$ is continuous, and $P_\calZ$ is single-valued on $(z_0 + \epsilon \bbB) \cap \calZ^\circ_\alpha$.

\cref{ass:local} allows us to formulate the following proposition which combines the existence of truncated solutions, the invariance of a neighborhood of $\calZ$, and equicontinuity (i.e. uniform Lipschitz continuity):

\begin{proposition}\label{prop:exist}
	Let \cref{ass:local} be satisfied for $z_0 \in \calZ$. Given any $T > 0$ and $K < \frac{\mu}{2 \alpha M}$, there exists a $(T, \epsilon)$-truncated solution $z$ for~\eqref{eq:sys_aw_gen} with $z(0) = z_0$ (where $\epsilon$ stems from \cref{ass:local}). Furthermore, $z$ satisfies, for almost all $t \in \dom z$,
	\begin{align*}
		z(t) \in \calZ + \tfrac{K M}{\mu} \bbB
		\qquad \text{and} \qquad
		\| \dot{z}(t) \| \leq \left(1 + \tfrac{\nu}{\mu} \right)M \, .
	\end{align*}

\end{proposition}

\begin{proof}
	First, we consider the existence of solutions: As mentioned, \cref{lem:vec_field} and \cref{prop:prox} imply that, on $(z_0 + \epsilon \bbB) \cap \calZ^\circ_\alpha$,~\eqref{eq:sys_aw_gen} reduces to a continuous ODE which is a well-posed inclusion (trivially). Hence, \cref{thm:filippov} guarantees the existence of a \emph{maximal} solution $z: [0, T'] \rightarrow (z_0 + \epsilon \bbB) \cap \calZ^\circ_\alpha$ starting at $z_0$ and with $x(T')$ on the boundary of $(z_0 + \epsilon \bbB) \cap \calZ^\circ_\alpha$.

	Next, by \cref{prop:prox}, we have $\nabla d_\calZ^2(z) = 2 (z - P_\calZ(z))$ for all $z \in \calZ^\circ_\alpha$. Hence, the Lie derivative of $d^2_\calZ$ along $F_K$ for all $z \in (z_0 + \epsilon \bbB) \cap \calZ^\circ_\alpha$ is well-defined and satisfies
	\begin{align*}
		\calL_{F_K} \left( \tfrac{1}{2} d^2_\calZ (z) \right)
		 & = (z - P_\calZ(z))^T \left( f(z, P_\calZ(z)) - \tfrac{1}{K} G^{-1}(P_\calZ(z))(z - P_\calZ(z) \right)                   \\
		 & \leq d_\calZ(z) \| f(z, P_\calZ(z))\| - \tfrac{1}{K} (z - P_\calZ(z))^T G^{-1}(P_\calZ(z)) \left(z - P_\calZ(z) \right) \\
		 & \leq d_\calZ(z) \| f (z, P_\calZ(z)) \| - \tfrac{\mu}{K} d^2_\calZ(z)
		= d_\calZ(z) ( M - \tfrac{\mu}{K} d_\calZ(z)) \, .
	\end{align*}
	It follows that $\calL_{F_K} \left( \tfrac{1}{2} d^2_\calZ (z ) \right) < 0$ whenever $d_\calZ(z) > \frac{K M}{\mu}$. Since $K < \frac{\mu}{2\alpha M}$ and using an invariance argument, it follows that $z(t)
		\in \calZ + \tfrac{K M}{\mu} \bbB	\subset \calZ^\circ_\alpha$ for all $t \in [0, T']$.

	In other words, for small enough $K$, any solution of~\eqref{eq:sys_aw_gen} starting at $z_0$ remains within a neighborhood of $\calZ$ on which the projection $P_\calZ$ is single-valued.

	Since $z(T')$ lies on the boundary of $(z_0 + \epsilon \bbB) \cap \calZ^\circ_\alpha$, but at the same time $z(T') \in \calZ + \frac{KM}{\mu} \bbB$, it follows that $\| z(T') - z_0 \| = \epsilon$. In other words, $z(T')$ lies on the boundary of $z_0 + \epsilon \bbB$ (rather than the boundary of $\calZ^\circ_\alpha$). Hence, (after restricting $z$ to $[0, T]$ if $T'> T$) it can be concluded that $z$ is a $(T, \epsilon)$-truncated solution of~\eqref{eq:sys_aw_gen}.

	Finally, we have that for all $z \in (\calZ + \frac{K M}{\mu} \bbB) \cap (z_0 + \epsilon \bbB)$ it holds that
	\begin{align*}
		\left\| \tfrac{1}{K} G(z)^{-1} (z - P_\calZ(z)) \right\| \leq \tfrac{1}{K} \nu \tfrac{K M}{\mu} \leq M\tfrac{\nu}{\mu } \, .
	\end{align*}
	It then follows from the definition of $M$ and the triangle inequality that $\| F_K(z) \| \leq M + M \tfrac{\nu}{\mu}$,	thus establishing the bound on $\| \dot z(t) \|$.
\end{proof}

The proof of \cref{prop:exist} suggests that the prox-regularity assumption on $\calZ$ is primarily required for $d_\calZ^2(z)$ to have a single-valued derivative in a neighborhood of $\calZ$. The following example shows, however, that prox-regularity is a more fundamental requirement which, in general, cannot be avoided.

\begin{example}\label{ex:pds_exist}
	Consider the set $\calZ := \{ (z_1, z_2) \in \bbR^2 \, | \, |z_2| \geq \max\{ 0, z_1\}^\kappa \}$ for any $\frac{1}{2} < \kappa < 1$. Further assume that $G(z) = \bbI$ and $f(z) = (1, 0)$ for all $z \in \bbR^n$. Hence, we can choose $M = \nu = \mu = 1$ and any $\epsilon > 0$ to satisfy \cref{ass:local}. Note, however, that $\calZ$ is not prox-regular at $(0,0)$. Namely, every point on the positive $z_1$-axis has a non-unique projection onto $\calZ$ as illustrated in \cref{fig:prx_counter}.

	\begin{figure}
		\centering
		\begin{subfigure}[t]{0.20\columnwidth}
			\includegraphics[width=\columnwidth]{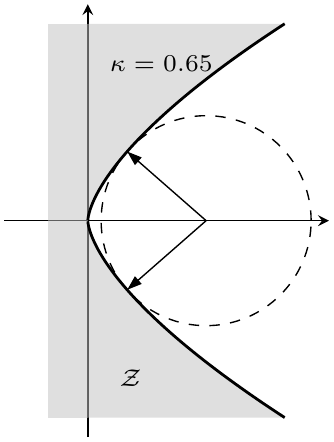}
			\caption{}\label{fig:prx_counter}
		\end{subfigure} \hspace{.15\columnwidth}
		\begin{subfigure}[t]{0.20\columnwidth}
			\includegraphics[width=\columnwidth]{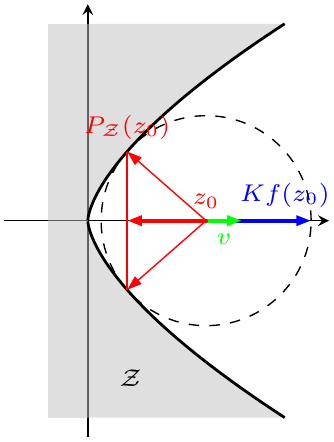}
			\caption{}\label{fig:prx_counter_kras}
		\end{subfigure} \hspace{.015\columnwidth}
		\vspace{-.7cm}
		\caption{Non-prox-regular set for \cref{ex:pds_exist}: (a) non-uniqueness of projection for every $(z_1, 0)$ (b) construction of Krasovskii regularization, namely $v \in \cocl F_K(z_0)$.}
	\end{figure}

	We claim that for every $K > 0$ there exists a Krasovskii\footnote{We cannot rely on the existence of Carath\'eodory solutions because $\calZ$ is not prox-regular and \cref{prop:exist} does not apply, but every Carath\'eodory solution (if it exits) is a Krasovskii solution.} solution (i.e., a solution of the inclusion $z \in \cocl F_K(z)$) starting on the $z_1$-axis that leaves the set $\calZ + \frac{K M}{\mu} \bbB$ established in \cref{prop:exist}.
	This can be deduced graphically from \cref{fig:prx_counter_kras}. Namely, let $z_0 = (z_{01},0)$ be such that $d_{\calZ} (z_0) = K$. Then, there exists $v = (v_1, 0)$ with $v_1 > 0$ in the Krasovskii-regularization of $F_K(z_0)$, i.e., $v \in \cocl F_K(z_0)$. In other words, on the boundary of $\calZ + K \bbB$, the vector $v$ points out of the (supposedly) invariant set. This, in turn, can be used to rigorously establish that the set $\calZ + K \bbB$ is not invariant, illustrating that the conclusion of \cref{prop:exist} does not hold without prox-regularity of $\calZ$, even when considering more general Krasovskii solutions.
\end{example}

\subsection{Anti-Windup Trajectories as Perturbed PDS}\label{subsec:perturb}

As a key technical result, we establish that solutions of the AWA \eqref{eq:sys_aw_gen} are also solutions of a $\sigma$-perturbation of the \pds in its alternate form~\eqref{eq:ivp_alt}. To prove this claim, consider $z_0 \in \calZ$, and let $M, \mu, \nu, \alpha, \epsilon > 0$ be such that \cref{ass:local} is satisfied. It follows from \cref{prop:pds_alt_form} that, for some $T>0$, every $(T, \epsilon)$-truncated solution $z: [0, T'] \rightarrow (z_0+ \epsilon \bbB)$ of the \pds \eqref{eq:aw_pds} with $z(0) = z_0$ is also a $(T, \epsilon)$-truncated solution of the inclusion
\begin{align}\label{eq:aw_sys_altform}
	\dot z \in \hat{F}(z) := f(z, P_\calZ(z)) - N^G_z \calZ \cap \gamma \bbB \quad \text{where} \quad \gamma := \max \left\{ \tfrac{M}{\sqrt{\mu}}, \tfrac{\nu}{\mu} M \right\}
\end{align}
and vice versa. This choice of $\gamma$ will be convenient in the proof of \cref{prop:perturbation} below. For now, note that using Cauchy-Schwarz, it holds that
\begin{align*}
	\sup_{z \in z_0 + \epsilon \bbB} \| f(z, P_\calZ(z)) \|_{G(z)} \leq \sup_{z \in z_0 + \epsilon \bbB} \underbrace{\sqrt{\| G(z) \|}}_{\leq 1/\sqrt{\mu}} \underbrace{\| f(z, P_\calZ(z)) \|}_{\leq M} \leq \gamma \, ,
\end{align*}
thus satisfying the condition on $\gamma$ in~\eqref{eq:ivp_alt} and \cref{prop:pds_alt_form}.

Furthermore, given $z_0 \in \calZ$, let \cref{ass:local} hold with some $\epsilon > 0$. By \cref{lem:vec_field} we have that $z \mapsto f(z, P_\calZ(z))$ is continuous on $z \in \calZ^\circ_\alpha$ and hence uniformly continuous on the bounded set $\calZ^\circ_\alpha \cap (z_0 + \epsilon \bbB)$. As a consequence of uniform continuity there exists $\omega \in \calK_\infty$ such that, for all $z, z' \in \calZ^\circ_\alpha \cap (z_0 + \epsilon \bbB)$, we have
\begin{align}\label{eq:f_cont}
	\| f(z, P_\calZ(z)) - f(z', P_\calZ(z')) \| \leq \omega( \| z - z' \|) \, .
\end{align}

\begin{proposition}\label{prop:perturbation}
	Consider $z_0 \in \calZ$ and let \cref{ass:local} hold with $M, \nu, \mu, \alpha$ and $\epsilon$. Further, let $K < \frac{\mu}{2 \alpha M}$. Then, for some $T >0$, every $(T, \epsilon)$-truncated solution $z:[0, T'] \rightarrow (z_0 + \epsilon \bbB)$ of~\eqref{eq:sys_aw_gen} is a solution of the $\sigma$-perturbation of~\eqref{eq:aw_sys_altform} with $\sigma := \max \left\{ \frac{K M}{\mu}, \omega \left(\frac{K M}{\mu} \right) \right\}$, where $\omega \in \calK_\infty$ satisfies~\eqref{eq:f_cont}.
\end{proposition}

\begin{proof}
	We need to show that the $(T, \epsilon)$-truncated solution $z$ satisfies
	\begin{align}\label{eq:infl_incl}
		\dot z(t) \in \tF_\sigma(z(t))\, , \qquad z(t) \in \calZ_\sigma
	\end{align}
	for almost all $t \in [0, T']$, where $\mathcal{Z}_\sigma := \calZ + \sigma\bbB$ and $
		\tF_\sigma(z) := \co \tF((z + \sigma \bbB) \cap \calZ) + \sigma \bbB$ for all $z \in \calZ_\sigma$ and with $\tF$ defined in~\eqref{eq:aw_sys_altform}.
	Note that for $z \in \calZ_\sigma$ we have that
	\begin{align}\label{eq:proj_prop}
		P_\calZ(z) \subset (z + \sigma \bbB) \cap \calZ\, .
	\end{align}

	\cref{prop:exist} guarantees that $z(t) \in \calZ + \frac{K M}{\mu} \bbB$, and since $\sigma \geq \frac{KM}{\mu}$ it follows that $z(t) \in \calZ_\sigma$ for all $t \in [0, T']$.
	For the remainder of the proof we omit the argument of $z(t)$ to simplify notation. All statements hold for almost all $t \in [0, T']$.

	Since $z - P_\calZ(z) \subset N_{P_\calZ(z)} \calZ$ for all $z \in \bbR^n$~\cite[Ex.~6.16]{rockafellarVariationalAnalysis2009} and using~\eqref{eq:norm_cone_equiv} we have
	\begin{align}\label{eq:trun_norm_prop1}
		\tfrac{1}{K} G(P_\calZ(z))^{-1} \left( z - P_\calZ(z) \right)
		\in N^G_{P_\calZ(z)} \calZ \, .
	\end{align}
	Furthermore, since $z \in \calZ + \frac{K M}{\mu}$ and using $\gamma$ as defined in~\eqref{eq:aw_sys_altform} we have that
	\begin{align}\label{eq:trun_norm_prop2}
		\left \| \tfrac{1}{K} G(P_\calZ(z))^{-1} \left( z - P_\calZ(z \right) \right\| \leq \tfrac{1}{K} \nu \tfrac{K M}{\mu} = \tfrac{\nu M}{\mu} \leq \gamma.
	\end{align}
	Combining~\eqref{eq:trun_norm_prop1} and~\eqref{eq:trun_norm_prop2} we have
	\begin{align}\label{eq:incl_pre}
		\dot z \in f (z, P_\calZ(z)) - N^G_{P_\calZ(z)}\calZ \cap \gamma \bbB \, .
	\end{align}
	Note that, in contrast to~\eqref{eq:aw_sys_altform}, the normal cone is evaluated at $P_\calZ(z)$.

	Next, using the fact that $\omega$, as defined in~\eqref{eq:f_cont}, is strictly increasing, and exploiting the definition of $\sigma$, we have
	\begin{align}\label{eq:lip_prop}
		\| f (z, P_\calZ(z)) - f (P_\calZ(z), P_\calZ(z)) \| \leq \omega\left(\| z - P_\calZ(z) \|\right) \leq \omega\left(\tfrac{KM}{\mu}\right) \leq \sigma\, .
	\end{align}

	Therefore, in summary, using~\eqref{eq:lip_prop} on~\eqref{eq:incl_pre}, as well as~\eqref{eq:proj_prop}, we have that
	\begin{align*}
		\dot z & \in f(z, P_\calZ(z))
		- N^G_{P_\calZ(z)}\calZ \cap \gamma \bbB                  \\
		       & \subset f (P_\calZ(z), P_\calZ(z)) + \sigma \bbB
		- N^G_{P_\calZ(z)}\calZ \cap \gamma \bbB                  \\
		       & = \tF(P_\calZ(z)) + \sigma \bbB
		\subset \tF((z + \sigma \bbB)\cap \calZ) + \sigma \bbB \subset \tF_\sigma(z) \, .
	\end{align*}
	Hence, $z(\cdot)$ satisfies~\eqref{eq:infl_incl} which completes the proof.
\end{proof}

\section{Uniform Convergence}\label{sec:unif}
We establish the graphical/uniform convergence of solutions of the anti-windup approximation~\eqref{eq:sys_aw_gen} to solutions of the projected dynamics~\eqref{eq:aw_pds}. This proof requires two arguments: On the one hand, we need to show that a graphically convergent sequence of solutions of~\eqref{eq:sys_aw_gen} converges to a solution of~\eqref{eq:aw_pds}. On the other hand, we need that such a graphically convergent sequence exists.

Starting with the latter requirement, we first recall that from a bounded sequence of sets, we can always extract a graphically convergent subsequence~\cite[Thm.~5.7]{goebelHybridDynamicalSystems2012}. This applies in particular to a sequence of (uniformly) truncated solutions:

\begin{lemma}\label{lem:subseq_extr}
	Consider a sequence $K_n \rightarrow 0^+$ and $z_{0} \in \calZ$. Given $T, \epsilon > 0$, any sequence $\{z_n\}$ of $(T, \epsilon)$-truncated solution of~\eqref{eq:sys_aw_gen} with $K = K_n$ and $z_n(0) = z_{0}$ has a graphically convergent subsequence.
\end{lemma}

\cref{lem:subseq_extr} is purely set-theoretic and does not imply that the limit $\gph \lim_{n \rightarrow \infty} z_n$ is a single-valued map. Hence, we need the following simplification\footnote{We require only the first of the two statements of the original theorem. Further, we consider the case where $\rho$ is constant. Finally, we work with truncated solutions which have, by definition, a compact domain (and thus are trivially \emph{locally eventually bounded}~\cite[Def.~5.24]{goebelHybridDynamicalSystems2012}).} of~\cite[Thm.~5.29]{goebelHybridDynamicalSystems2012}:

\begin{lemma}\label{lem:unif_cvg}
	Let the inclusion~\eqref{eq:gen_inclusion} be well-posed and $z_0 \in \calZ$.
	Further, given any $T, \epsilon, \rho > 0$ and $\delta_i \rightarrow 0^+$, let $z_i: [0, T_i] \rightarrow \calX_i$ denote a $(T, \epsilon)$-truncated solution of the $\delta_i \sigma$-perturbation of~\eqref{eq:gen_inclusion}.
	If the sequence $\{z_i\}$ converges graphically, then convergence is to a solution $z: [0, T] \rightarrow \calX$ of~\eqref{eq:gen_inclusion}, where $T = \lim_{i \rightarrow \infty} T_i$.
\end{lemma}

\begin{remark}\label{rem:graph_unif_cvg}
	In the context of \cref{lem:unif_cvg}, graphical convergence implies uniform convergence to $z$ on every subinterval of $[0, T)$~\cite[Lem.~5.28]{goebelHybridDynamicalSystems2012}. Furthermore, if $T_i \geq T$ for all $i$, then convergence is uniform on $[0, T]$.
\end{remark}

Since, by \cref{prop:perturbation}, solutions of~\eqref{eq:sys_aw_gen} are solutions of a $\sigma$-perturbation of an alternate form PDS~\eqref{eq:aw_sys_altform} we can use \cref{lem:unif_cvg} to establish the following result:

\begin{proposition}\label{prop:incl}
	Given $z_0 \in \calZ$, let \cref{ass:local} be satisfied. Consider $T>0$ and a sequence $K_i \rightarrow 0^+$, and assume that a sequence of $(T, \epsilon)$-truncated solutions $z_i$ of the AWA~\eqref{eq:sys_aw_gen} with $K = K_i$ and $z_i(0) = z_0$ for all $i$ converges graphically. Then, the limit is a $(T, \epsilon)$-truncated solution of the PDS~\eqref{eq:aw_pds}.
\end{proposition}

\begin{proof}
	Let $M, \mu, \nu >0$ and $\omega \in \calK_\infty$ be defined as in \cref{ass:local} and~\eqref{eq:f_cont}, respectively. Using \cref{lem:kinf}, there exist $\sigma_1, \sigma_2 \in \calK_\infty$ such that $\omega(rs) \leq \sigma_1(r) \sigma_2(r)$ for all $r, s \geq 0$. Hence, we define $\delta_i := \max\{K_i, \sigma_1(K_i)\}$ and $\rho := \max\left\{ \tfrac{M}{\mu}, \sigma_2 \left(\tfrac{M}{\mu} \right) \right\}$.

	\cref{prop:perturbation} states that for every $K_i$, the solution $z_i$ of~\eqref{eq:sys_aw_gen} is also a solution of the $\sigma$-perturbation of~\eqref{eq:aw_sys_altform} with $\sigma := \max \left \{K_i \tfrac{M}{\mu}, \omega \left(K_i \tfrac{M}{\mu} \right) \right\}$. It follows that $z_i$ is also a solution of every $\sigma'$-perturbation of~\eqref{eq:aw_sys_altform} with $\sigma' \geq \sigma$. In particular, we can set
	\begin{align*}
		\sigma' := \delta_i \rho = \max\{K_i, \sigma_1(K_i)\} \max\left\{ \tfrac{M}{\mu}, \sigma_2(\tfrac{M}{\mu}) \right\} \geq \sigma \, ,
	\end{align*}
	and thus we have that $z_i$ is a solution of the $\delta_i \rho$-perturbation of~\eqref{eq:aw_sys_altform}.

	Since, by assumption, $\{z_i\}$ converges graphically to $z$ it follows from \cref{lem:unif_cvg} that $z$ is a solution of~\eqref{eq:aw_sys_altform}, and, by \cref{prop:pds_alt_form}, $z$ is a solution of~\eqref{eq:aw_pds}.

	Finally, we need to show that $z: [0, T'] \rightarrow (z_0 + \epsilon \bbB)$ is a $(T, \epsilon)$-truncated solution. Namely, we need to show that either $T = T'$ or $\| z(T) - z_0 \| = \epsilon$. This requirement is equivalent to $(T', z(T'))$ lying on the boundary of the cylinder $\calX := [0, T] \times (z_0 + \epsilon \bbB)$. Since, by definition, for every $i$, $z_i$ is a $(T, \epsilon)$-truncated solution of~\eqref{eq:sys_aw_gen} we have that $(T_i, z_i(T_i)) \in \partial \calX$ for all $i$. Since $\partial \calX$ is closed, it follows that the limit also lies in $\partial \calX$.
\end{proof}

Now, we can immediately combine \cref{lem:subseq_extr} and \cref{prop:incl} to arrive at our first main result about the graphical convergence of truncated solutions (i.e., local) solutions of anti-windup approximations to a projected dynamical system:

\begin{theorem}\label{thm:loc}
	Let \cref{ass:local} be satisfied for some $z_0 \in \calZ$. Given any $T>0$ (and $\epsilon > 0$ from \cref{ass:local}), consider a sequence $K_n \rightarrow 0^+$ and let $\{z_n \}$ denote a sequence of $(T, \epsilon)$-truncated solutions of the AWA~\eqref{eq:sys_aw_gen} with $K = K_n$ and $z_n(0)~=~z_0$.
	Then, there exists a subsequence of $\{ z_{n}\}$ that converges graphically to a $(T,\epsilon)$-truncated solution of the PDS~\eqref{eq:aw_pds}.
\end{theorem}

Under certain circumstances, it can be useful to know that, rather than a subsequence of gains $\{K_n\}$, any sequence $K_n \rightarrow 0^+$ will lead to a converging sequence of solutions. This is guaranteed if it is known that the \pds~\eqref{eq:aw_pds} has a unique solution:

\begin{corollary}\label{cor:uniq}
	Let \cref{ass:local} be satisfied for some $z_0 \in \calZ$. Given any $T>0$ (and $\epsilon > 0$ from \cref{ass:local}), assume that the PDS~\eqref{eq:aw_pds} admits a unique $(T, \epsilon)$-truncated solution $z$ with $z(0) = z_0$.
	Then, any sequence $\{z_n \}$ of $(T, \epsilon)$-truncated solutions of the AWA~\eqref{eq:sys_aw_gen} with $z_n(0)~=~z_0$ and $K = K_n$ with $K_n \rightarrow 0^+$ converges graphically to the (unique) $(T,\epsilon)$-truncated solution of the PDS~\eqref{eq:aw_pds}.
\end{corollary}

\begin{proof}
	Assume, for the sake of contradiction, that $\{z_n\}$ does not converge to the unique solution $z$ of~\eqref{eq:aw_pds}. This implies that there exists $\nu > 0$ and a subsequence $\{z_m\}$ of $\{z_n\}$ such that $d_\infty(\gph z_m, \gph z) \geq \nu$ for all $m$ where $d_\infty$ denotes the Hausdorff distance between two sets. (In particular, since $z$ is a truncated solution $\gph z$ is compact and thus graphical convergence is equivalent to convergence with respect to $d_\infty$~\cite[Ex.~4.13]{rockafellarVariationalAnalysis2009}.) However, by \cref{lem:subseq_extr}, the sequence $\{z_m\}$ has a convergent subsequence that converges to some limit $\tilde{z}$. By \cref{prop:incl}, $\tilde{z}$ is a solution of~\eqref{eq:aw_pds}, but we also have $\| \tilde{z} - z \|_\infty \geq \varepsilon$ which contradicts the uniqueness of $z$.
\end{proof}

Finally, we can state the following \emph{ready-to-use} result about uniform convergence in the case when the existence of unique complete solutions is guaranteed:

\begin{corollary}
	Consider the AWA~\eqref{eq:sys_aw_gen}, let $\calZ$ be prox-regular, $f$ globally Lipschitz, and there exist $\mu, \nu > 0$ such that $\mu \bbI \preceq G^{-1}(z) \preceq \nu \bbI $ for all $z \in \bbR^n$.
	Given $z_0 \in \calZ$ and a sequence $K_n \rightarrow 0^+$, every sequence of complete solutions $z_n$ of the AWA~\eqref{eq:sys_aw_gen} with initial condition $z_0$ and $K = K_n$ converges uniformly to the unique complete solution of the PDS~\eqref{eq:aw_pds} on every compact interval $[0, T]$.
\end{corollary}

\begin{proof}
	Note that the assumptions on $\calZ$, $G$, and $f$ guarantee that for every initial condition~\eqref{eq:aw_pds} admits a unique complete (Carath\'eodory) solution (\cref{thm:pds_exist}).

	Hence, given any $T > 0$, let $z:[0,T] \rightarrow \calZ$ denote the unique solution of the \pds~\eqref{eq:aw_pds} and define $\epsilon > \sup_{t \in [0, T]} \| x(t) - x_0 \|$. Since $f$ is continuous and hence bounded over a compact set, \cref{ass:local} is satisfied with $\nu, \mu, \alpha$ and by choosing $M:= \max_{z \in z_0 + \epsilon \bbB} \|f(z, P_\calZ(z)) \|$.
	\cref{thm:loc} guarantees convergence of a subsequence to the $(T, \epsilon)$-truncated solution $z: [0, T'] \rightarrow \calX$ of~\eqref{eq:aw_pds}. Moreover, for the same reason as in \cref{cor:uniq} the sequence itself converges.

	Finally, by definition of $\epsilon$, we have that $z$ is defined on $[0, T']$ with $T' = T$ and $\|z(T) - z_0\| < \epsilon$ and, in this case, graphical convergence of $(T, \epsilon)$-truncated solutions implies their uniform convergence on $[0, T]$ (see \cref{rem:graph_unif_cvg}).
\end{proof}

\begin{remark}
	\cref{thm:loc} and its corollaries can be slightly generalized, albeit at the expense of additional technicalities. For instance, instead of considering a single initial condition $z_0 \in \calZ$, it is in general possible to consider a sequence of initial conditions (under some additional restrictions) that converges to $z_0$.
\end{remark}

\section{Semiglobal Practical Robust Stability}\label{sec:rob}

Since anti-windup approximations can be seen as perturbations of projected dynamical systems, we can establish semiglobal practical asymptotic stability in $K$ with the following simplified\footnote{We consider only global asymptotic stability, which allows us to use the distance function instead of more general indicator functions. Further, we limit ourselves to $\rho$ being a positive constant instead of a function. As noted in \cref{rem:pre_stab}, compactness and stability of $\calA$ guarantee the existence of complete solutions since finite-time escape is not possible.} lemma:

\begin{lemma}\cite[Lem.~7.20]{goebelHybridDynamicalSystems2012}\label{lem:pract_robust}
	Let the inclusion~\eqref{eq:gen_inclusion} be well-posed and let $\calA \subset \calX$ be a compact and asymptotically stable set for~\eqref{eq:gen_inclusion}, i.e., $d_\calA( x (t)) \leq \beta( d_\calA(x(0)), t)$ for all $t \geq 0$	holds for some $\beta \in \calKL$ and any (complete) solution $x$ of~\eqref{eq:gen_inclusion}. Then, for every $\rho > 0$, every compact $\mathcal{B} \subset \bbR^n$, and every $\zeta > 0$ there exists $\delta \in (0,1)$ such that every solution $x_\delta$ of the $\delta\rho$-perturbation of~\eqref{eq:gen_inclusion} starting in $\mathcal{B} \cap C_{\delta \rho}$ satisfies $d_\calA(x_\delta (t)) \leq \beta( d_\calA( x_\delta(0)), t) + \zeta$ for all $t \geq 0$.
\end{lemma}

Hence, using \cref{prop:perturbation}, we arrive at the following second main result:

\begin{theorem}\label{thm:rob_stab}
	Consider a \pds~\eqref{eq:aw_pds} where $\calC$ is Clarke regular, $f$ and $G$ are continuous, and for which the compact set $\calA \subset \calZ$ is globally asymptotically stable, i.e., there is $\beta \in \calKL$ such that for every solution $z$ it holds that
	\begin{align*}
		d_\calA( z (t)) \leq \beta( d_\calA(z(0)), t)  \quad \forall t \geq 0 \, .
	\end{align*}
	Then, for every $\zeta > 0$ and every compact $\calB \subset \calZ$ there exists $K^\star > 0$ such that for all $K \in (0, K^\star)$ every solution $z_K$ of the AWA~\eqref{eq:sys_aw_gen} with $z_K(0) \in \calB$ satisfies
	\begin{align*}
		d_\calA( z_K(t)) \leq \beta( d_\calA(z_K(0)), t)  + \zeta \qquad \forall t \geq 0\, .
	\end{align*}
\end{theorem}

\begin{proof}
	First, we establish that \cref{ass:local} holds for every $z_0 \in \calB$. Since $\calB$ is compact, let $\overline{\beta} := \max_{z \in \calB} \beta(d_\calA(z), 0)$. Since $\beta$ is strictly increasing and unbounded, and, since $\calA$ is compact, the set $\calV := \{ z \, | \, \beta(d_\calA(z), 0) \leq \overline{\beta}\}$ is compact. Hence, we can choose $\epsilon > 0$ such that $\calV \subset \calB + \epsilon \bbB$. It follows that any solution of
	~\eqref{eq:aw_pds} starting in $\calB$ remains in $\calB + \epsilon \bbB$. By continuity over the compact set $\calB + \epsilon \bbB$, we can further choose and $M, \mu, \nu > 0$ such that
	$\| f(z, P_\calZ(z)) \| \leq M$ and $\mu \bbI \preceq G^{-1}(z) \preceq \nu \bbI$ holds for all $z \in \calB + \epsilon \bbB$. Thus, \cref{ass:local} is satisfied for all $z_0 \in \calB$. Further, every (complete) solution of the PDS~\eqref{eq:aw_pds} starting in $\calB$ remains in $\calB + \epsilon \bbB$ and hence can be written in its alternate form~\eqref{eq:aw_sys_altform}.
	Next, fix any $\rho > 0$. \cref{lem:pract_robust} implies that for every $\zeta > 0$ and every compact $\calB \subset \calZ$ there exists $\delta \in (0,1)$ such that the $\delta \rho$-perturbation is $\zeta$-practically pre-asymptotically stable.
	Given such a $\delta$, we conclude that there exists $K^\star > 0$ that, for all $K' < K^\star$,
	$\max \{ K' \tfrac{M}{\mu}, \omega (K' \tfrac{M}{\mu} ) \} \leq \delta \rho$ since $\omega$ is strictly increasing and $\omega(0) = 0$.
	Thus, \cref{prop:perturbation} states that the solution of~\eqref{eq:sys_aw_gen} with $K = K'$ is a solution of the $\sigma$-perturbation of~\eqref{eq:aw_sys_altform} with $\sigma = \max \{ K' \tfrac{M}{\mu}, \omega(K' \tfrac{M}{\mu}) \}$. Moreover, it is also solution to any $\sigma'$-perturbation with $\sigma' \geq \sigma$ and, in particular, for $\sigma' = \delta \rho$.
\end{proof}

Since the asymptotic stability of $\calA$ can often be established with a smooth Lyapunov function (see \cite[Thm.~3.18]{goebelHybridDynamicalSystems2012}), we can also state the following corollary:

\begin{corollary}
	Consider the PDS~\eqref{eq:aw_pds} where $\calC$ is Clarke regular, $f$ and $G$ are continuous. Further, consider a compact set $\calA \subset \calZ$ for which there exists a  \emph{Lyapunov function}\footnote{Namely, $V : \bbR^n \rightarrow \bbR_{\geq 0}$ is a Lyapunov function for $\calA$ if it is differentiable everywhere on $\calZ$, there exist $\underline{\alpha}, \overline{\alpha} \in \calK_\infty$ such that $\underline{\alpha}(d_\calA(z)) \leq V(z) \leq\overline{\alpha}(d_\calA(z))$ for all $z \in \calZ$, and $\left\langle \nabla V(z), \Pi_\calZ^G [f](z) \right\rangle \leq - \alpha(z)$ for all $z \in \calZ$ where $\alpha: \bbR^n \rightarrow \bbR_{\geq 0}$  is continuous and positive definite with respect to $\calA$, i.e., $\alpha(z) > 0$ for all $z \notin \calA$ and $\alpha(z) = 0$ for all $z \in \calA$.}.
	Then, for every $\zeta > 0$ and every compact set $\calB \subset \bbR^n$, there exists $K^\star$ such that for all $K \in  (0, K^\star)$ every solution of~\eqref{eq:sys_aw_gen} converges to a subset of $\calA  + \zeta \bbB$.
\end{corollary}

\section{Preservation of Equilibria \& Robust Convergence}\label{sec:rob_conv}

Finally, we consider the special case of~\eqref{eq:sys_aw_gen} when $f$ depends only on $P_\calZ(z)$, i.e., we study the system
\begin{align}\label{eq:sys_aw_gen_spec}
	\dot z \in F_K(z) := f(P_\calZ(z)) - \tfrac{1}{K} G^{-1}(P_\calZ(z)) ( z - P_\calZ(z))
\end{align}
where, as before, $\calZ$ is an $\alpha$-prox-regular set, $G: \calZ \rightarrow \bbS_+^n$ is a continuous metric, $K > 0$ is a scalar, and $f: \calZ \rightarrow \bbR^n$ is a continuous vector field.
All of the previous results for~\eqref{eq:sys_aw_gen} also apply to~\eqref{eq:sys_aw_gen_spec}. In particular, as $K \rightarrow 0^+$, trajectories of~\eqref{eq:sys_aw_gen_spec} converge uniformly to solutions of the \pds~\eqref{eq:aw_pds}. Also, the practical stability results of \cref{sec:rob} apply, but we show next that a stronger result can be derived for~\eqref{eq:sys_aw_gen_spec}.

In the following, $z^\star$ is a \emph{weak equilibrium} of~\eqref{eq:sys_aw_gen_spec} if the constant trajectory $z \equiv z^\star$ is a solution of~\eqref{eq:sys_aw_gen_spec}. Since we consider only Carath\'eodory solutions, $z^\star$ is a weak equilibrium of~\eqref{eq:sys_aw_gen_spec} if and only if $0 \in F_K(z^\star)$.

An important advantage of~\eqref{eq:sys_aw_gen_spec} over the more general system~\eqref{eq:sys_aw_gen} is that equilibria of~\eqref{eq:aw_pds} are preserved in the following sense (which generalizes~\cite[Prop.~4]{hauswirthImplementationProjectedDynamical2020}):

\begin{proposition}\label{prop:equil}
	If $\o{z}^\star \in \calZ$ is a weak equilibrium point of the PDS~\eqref{eq:aw_pds}, then there exists $K^\star > 0$ such that for all $K \in (0, K^\star)$ there exists a weak equilibrium point ${z}_K^\star \in \o{z}^\star + N_{\o{z}^\star} \calZ \cap \tfrac{1}{2\alpha} \interior \bbB$ for the AWA~\eqref{eq:sys_aw_gen_spec}.
	Conversely, if ${z}^\star_K \in \calZ^\circ_\alpha$ is a weak equilibrium of~\eqref{eq:sys_aw_gen_spec} for some $K$, then $P_\calZ(z^\star)$ is a weak equilibrium of~\eqref{eq:aw_pds}.
\end{proposition}

\begin{proof}
	Given a weak equilibrium $\o{z}^\star \in \calZ$ of~\eqref{eq:aw_pds}, let $z^\star_K := \o{z}^\star - K G(\o{z}^\star)f(\o{z}^\star)$. For $K \in (0, K^\star) := 1/ (2\alpha \| G(\o{z}^\star) f(\o{z}^\star)\|)$, we have $z^\star_K  \in \calZ^\circ_\alpha$.

	Since $\o{z}^\star$ is an equilibrium of~\eqref{eq:aw_pds} (by assumption) and using \cref{lem:pds_decomp}, we have $f(\o{z}^\star) \in - N^G_{\o{z}^\star} \calZ$. It follows from~\eqref{eq:norm_cone_equiv} that $-K G(\o{z}^\star) f(\o{z}^\star) \in N_{\o{z}^\star} \calZ$ and consequently $z^\star_K \in \o{z}^\star + N_{\o{z}^\star} \calZ$.
	By \cref{prop:prox}, it follows that $P_\calZ(z_K^\star) = \o{z}^\star$ and therefore
	\begin{align*}
		F_K(z^\star_K) = f (\o{z}^\star) - \tfrac{1}{K}G^{-1}(\o{z}^\star) \left(\o{z}^\star - K G(\o{z}^\star)f(\o{z}^\star) - \o{z}^\star \right) = 0 \, .
	\end{align*}
	Thus, $z^\star_K$ is a weak equilibrium of~\eqref{eq:sys_aw_gen_spec}. The converse case follows the same ideas.
\end{proof}

Although equilibria of the PDS~\eqref{eq:aw_pds} are preserved by the AWA~\eqref{eq:sys_aw_gen_spec} (after projection), it is not clear whether convergence properties are preserved, especially since we are primarily interested in the convergence of $t \mapsto P_\calZ(z(t))$ rather than the convergence of the solution $z$ of~\eqref{eq:sys_aw_gen_spec}. \cref{thm:rob_stab} suggests that, in general, convergence is only within a neighborhood of asymptotically stable equilibria of the PDS~\eqref{eq:aw_pds}.

However, as we shown below, under additional conditions on $f, G$ and $\calZ$, the projected solutions $t \mapsto P_\calZ(z(t))$ do indeed converge to an equilibrium of~\eqref{eq:aw_pds}.

\subsection{Anti-Windup Approximations of Monotone Dynamics}
Next, we show that if $-f$ is monotone and $G \equiv \bbI$, then $F_K$, as defined in~\eqref{eq:sys_aw_gen_spec}, is monotone for small enough $K$. This, in turn, allows us conclude asymptotic stability of~\eqref{eq:sys_aw_gen_spec}.

Since we require only monotonicity of $f$, the following results can be used not only when $f$ is chosen as the gradient of a convex cost function, but also for saddle-point flows (see \cref{subsec:saddle}), and pseudo-gradients for Nash-equilibrium seeking~\cite{nagurneyProjectedDynamicalSystems1996,depersisDistributedaveragingintegral2019}.

Given a set $\calC \subset \bbR^n$, recall that a map $F: \calC \rightrightarrows \bbR^n$ is (strictly; $\beta$-strongly) \emph{monotone} if for all $x, x' \in \calC$ and all $v \in F(x)$ and $v' \in F(x')$ it holds that
\begin{align*}
	\left\langle v - v', x - x' \right\rangle \geq 0 \, (> 0; \geq \beta \| x- x' \|^2) \ .
\end{align*}
Further, if $\calC$ is $\alpha$-prox-regular, the map $x \mapsto N_x \calC$ has a \emph{hypomonotone localization}~\cite[Ex.~13.38]{rockafellarVariationalAnalysis2009}, i.e., for all $x, x' \in \calC$, all $\eta \in N_x \calC \cap \bbB$, and all $\eta' \in N_{x'} \calC \cap \bbB$ we have
\begin{align*}
	\left\langle \eta - \eta', x- x' \right\rangle \geq - 2 \alpha \| x - x' \|^2 \, .
\end{align*}
In particular, if $\calC$ is convex, we have $\left\langle \eta' - \eta, x'- x \right\rangle \geq 0$ and $x \mapsto N_x \calC$ is monotone.

\begin{proposition}\label{prop:monot}
	Consider $F_K$ as defined in~\eqref{eq:sys_aw_gen_spec} with $G \equiv \bbI$ and $\calC$ is assumed to be $\alpha$-prox-regular. Let $-f$ be $\beta$-strongly monotone and globally $L$-Lipschitz. Then $-F_K$ is strictly monotone on $\calZ^\circ_\alpha$ for all $0< K < 4 (\beta - 2 \alpha)/ L^2$.
\end{proposition}

\begin{proof}
	Given any $z, z' \in \calZ^\circ_\alpha$, let $\o{z} := P_\calZ(z)$ and $\o{z}' := P_\calZ(z')$. Further, let $\eta := z - \o{z} \in N_{\o{z}} \calZ$ and $\eta' := z' - \o{z}' \in N_{\o{z}'} \calZ$. We can work directly with the monotonicity of $f$, the hypomonotocity of $z \mapsto N_z \calZ$, and Cauchy-Schwarz to derive
	\begin{align*}
		\left\langle z - z', F_K(z) - F_K(z') \right\rangle
		 & = \left\langle z - z', f(\o{z}) - f(\o{z}') - \tfrac{1}{K}(z - \o{z}) + \tfrac{1}{K}(z' - \o{z}')) \right\rangle                          \\
		 & = \left\langle \o{z} - \o{z}' + \eta - \eta', f(\o{z}) - f(\o{z}') - \tfrac{1}{K} (\eta - \eta') \right\rangle                            \\
		 & = \left\langle \o{z} - \o{z}', f(\o{z}) - f(\o{z}') \right\rangle   - \tfrac{1}{K} \left\langle  \eta - \eta', \eta - \eta' \right\rangle
		\\
		 & \qquad \qquad
		+ \underbrace{\left\langle \eta - \eta', f(\o{z}) - f(\o{z}') \right\rangle}_{L \| \eta - \eta'\| \| \o{z} - \o{z}' \|}
		- \tfrac{1}{K} \underbrace{\left\langle \o{z} - \o{z}', \eta - \eta' \right\rangle}_{\geq - 2 \alpha  \| \o{z} - \o{z}' \|^2}                \\
		 & \leq  - (\beta - 2 \alpha) \| \o{z} - \o{z}' \|^2 + L \| \o{z} - \o{z}' \| \| \eta - \eta' \|  - \tfrac{1}{K} \| \eta - \eta' \|^2 \, .
	\end{align*}

	A sufficient condition for the righthand side to be negative for all $\o{z} \neq \o{z}'$ is that $\beta - 2\alpha > 0$ and that the determinant $\tfrac{1}{K}(\beta  - 2\alpha) - \tfrac{1}{4}L^2$ is positive, i.e., if $0< K < 4(\beta - 2 \alpha)/L^2$.
\end{proof}

This leads us to our third theoretical result which establishes convergence of anti-windup approximations for strongly monotone dynamics on convex sets:
\begin{theorem}\label{thm:monot}
	Consider the AWA~\eqref{eq:sys_aw_gen_spec} with $G \equiv \bbI$ and let $\calC$ be closed convex. Assume that $-f$ is $\beta$-strongly monotone and globally $L$-Lipschitz. Then, for all $K < 4 \beta / L^2$, every trajectory of~\eqref{eq:sys_aw_gen_spec} converges to an equilibrium point $z^\star$ (which is unique) such that $P_\calZ(z^\star)$ is the unique equilibrium of the \pds~\eqref{eq:aw_pds}.
\end{theorem}

\begin{proof}

	Because of convexity of $\calZ$, $P_\calZ(z)$ is single-valued and continuous for all $z \in \bbR^n$ and globally 1-Lipschitz (i.e., non-expansive). As a consequence, $F_K$ is globally Lipschitz continuous and there exists a unique complete solution of~\eqref{eq:sys_aw_gen_spec} for every initial condition $z(0) \in \bbR^n$. Furthermore, since $K < 4 \beta / L^2$ and $\calZ$ is convex (which lets us take $\alpha \rightarrow 0^+$), \cref{prop:monot} guarantees that $F_K$ is strictly monotone on $\bbR^n$.

	Next, recall that the strong monotonicity of $-f$ and convexity of $\calZ$ imply that \eqref{eq:aw_pds} has a unique equilibrium $\o{z}^\star$~\cite[Thm.~2.3]{nagurneyProjectedDynamicalSystems1996}. Consequently, \cref{prop:equil} guarantees the existence of an equilibrium point $z^\star$ of~\eqref{eq:sys_aw_gen_spec} such that $P_\calZ(z^\star) =  \o{z}^\star$. Furthermore, $z^\star$ is unique by~\cite[Thm.~2.2]{nagurneyProjectedDynamicalSystems1996}.  In particular, strict monotonicity of $F_K$ implies that $V(z):= \tfrac{1}{2} \| z - z^\star \|^2$ is a Lyapunov function for~\eqref{eq:sys_aw_gen_spec} which can be used to establish global asymptotic stability of~$z^\star$.
\end{proof}

\cref{thm:monot} can, presumably, be generalized to prox-regular sets as well as general metrics $G$. However, in that case, additional restriction on $z(0)$ are required, the threshold value for $K$ is less easily quantifiable, and convergence is likely only local.

\section{Application: Anti-Windup for Autonomous Optimization}\label{sec:aw_appl}
Next, we show how the AWA~\eqref{eq:sys_aw_gen} models physical systems and how anti-windup implementations can be used in the context of autonomous optimization to approximate closed-loop optimization dynamics that are formulated as projected dynamical systems.

First, consider the feedback control loop illustrated in \cref{fig:aw1}. Namely, we study a plant controlled by an integral feedback controller that is subject to input saturation modelled as an Euclidean projection. An anti-windup scheme is in place to avoid integrator windup. More precisely, we consider a dynamical system of the form
\begin{subequations}\label{eq:fb_aw_sys}
	\begin{align*}
		\dot x & \in \tilde{f}(x, P_\calU(u))                                                                  &  & x \in \bbR^m \\
		\dot u & \in k(x, u, P_\calU(u)) -  \tfrac{1}{K} \tilde{G}^{-1}(P_\calU(u)) ( u - P_\calU(u) )  \qquad &  & u \in \bbR^p
	\end{align*}
\end{subequations}
where $\calU \subset \bbR^p$ is prox-regular, $\tilde{f}:\bbR^m \times \calU \rightarrow \bbR^m $ and $k: \bbR^m \times \bbR^p \times \calU \rightarrow \bbR^p$ are continuous vector fields, $\tilde{G}: \calU \rightarrow \bbS_p^+$ is a continuous metric, and $K>0$.

\begin{figure}[bt]
	\centering
	\begin{tikzpicture}
		\matrix[ampersand replacement=\&, row sep=0.3cm, column sep=.55cm] {

			\& \& \node[block] (awgain) {$\frac{1}{K} \tilde{G}(\overline{u})$}; \&
			\node[smallsum] (awsum) {};    \\

			\& \node[smallsum] (aw_inf) {};
			\& \node[block] (int) {$\int$};
			\& \node[branch](br1) {};
			\& \node[saturation block, minimum height=2.7em,minimum width=2.7em](sat){};
			\& \node[branch](br2){}; \\

			\node[block] (ctrl) {$k(\cdot, u, \overline{u})$};
			\& \&  \& \& \& \&
			\node[block] (plant) {
				\begin{tabular}{c}
					$\dot x = \tilde{f}(x, \cdot)$ \\
				\end{tabular}
			};    \\

			\node[none](edge) {};    \\
		};

		\draw[connector] (ctrl.north)|-(aw_inf.west) node[at end, below]{$+$};
		\draw[connector] (awgain.west)-|(aw_inf.north) node[near end, left]{$-$};
		\draw[connector] (aw_inf.east)--(int.west);
		\draw[line] (int.east)--(br1.west) node[at end, below] {$u$};
		\draw[connector] (br1.east)--(sat.west);
		\draw[line] (sat.east)--(br2.west);
		\draw[connector] (br2.east)-|(plant.north) node[near start, above]{$\overline{u} := P_\calU(u)$};
		\draw[line] (plant.south)|-(edge.center);
		\draw[connector] (edge.center)--(ctrl.south);

		\draw[connector] (br2.north)|-(awsum.east) node[at end, above] {$-$};
		\draw[connector] (br1.north)--(awsum.south) node[near end, right] {$+$};
		\draw[connector] (awsum.west)--(awgain.east);
	\end{tikzpicture}
	\caption{Feedback loop with anti-windup (dependence of $k$ and $\tilde{G}$ on $u, \overline{u}$ is not drawn)}\label{fig:aw1}
\end{figure}
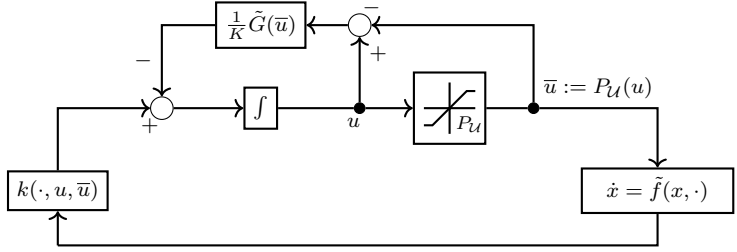

The system~\eqref{eq:fb_aw_sys} can be brought into the form of an AWA~\eqref{eq:sys_aw_gen} with $n = m + p$ by defining $z := \begin{smallbmatrix} x \\ u \end{smallbmatrix}$, $\calZ := \bbR^m \times \calU$, and $G(z) := \begin{smallbmatrix} \bbI & 0 \\ 0 & \tilde{G}(u) \end{smallbmatrix}$.
Thus, we further have
\begin{align*}
	P_\calZ(z) = \begin{bmatrix} x \\ P_\calU(u) \end{bmatrix} \qquad \text{and} \qquad
	f(z, P_\calZ(z)) := \begin{bmatrix} \tilde{f}(x, P_\calU(u)) \\ k(x, u, P_\calU(u))
	\end{bmatrix} \,.
\end{align*}
With these definitions, the \pds \eqref{eq:aw_pds} takes the form
\begin{align*}
	\dot x & = \tilde{f}(x, u)                              &  & x \in \bbR^m     \\
	\dot u & = \Pi_\calU^{\tilde{G}} [k(x, u, u)](u) \qquad &  & u \in \calU \, ,
\end{align*}
where we can ignore the projection onto $\calU$ in the third argument of $k$, because any solution of the PDS~\eqref{eq:aw_pds} is viable (i.e., remains in $\calU$) by definition.

\begin{remark}
	\cref{fig:aw1} shows one limitation of our problem setup: Compared to existing work on anti-windup control~\cite{zaccarianModernAntiwindupSynthesis2011,tarbouriechAntiwindupDesignOverview2009}, we do not model any proportional controller subject to input saturation. This is motivated, on one hand, by theoretical necessity. On the other hand, for our application scenario of autonomous optimization discussed below, stability of the physical plant is usually a prerequisite.
\end{remark}

\subsection{Feedback-based Gradient Schemes for Quadratic Programs}
To illustrate the design opportunities for autonomous optimization, we present three anti-windup schemes that approximate projected gradient flows for a quadratic program (QP). We consider the relatively simple problem of solving a QP as it allows for a concise presentation, easy implementation, and comparability. However, needless to say, our theoretical results in the previous sections cover much more general setups.

Our goal is to design a feedback controller that steers a plant to a steady state that solves the optimization problem
\begin{align}\label{eq:app_opt_prob}
	\begin{split}
		\begin{aligned}
			\text{minimize} \quad   & \Phi(x)  :=\tfrac{1}{2} x^T Q x + c^T x + d       \\
			\text{subject to} \quad & x = h(u)        :=  H u + w \\
			                        & u \in \calU     := \{ v \, | \, A_u v \leq b_u \}
		\end{aligned}
	\end{split}
\end{align}
where $x \in \bbR^m$ and $u \in \bbR^p$ denote the system state and control input, respectively, and $Q \in \bbS_+^m$, $A_u \in \bbR^{r \times p}$ and the remaining parameters are of appropriate size.
The map $h$ denotes the steady-state input-to-state map of the plant subject to the disturbance $w$.\footnote{
In contrast to~\eqref{eq:fb_aw_sys}, we assume for~\eqref{eq:app_opt_prob} that the physical plant is described by an steady-state input-to-state map $x = h(u)$ that satisfies $\tilde{f}(h(u), u) = 0$ for all $u \in \calU$. This approximation can be motivated by singular perturbation ideas~\cite{mentaStabilityDynamicFeedback2018,hauswirthTimescaleSeparationAutonomous2019} which stipulate that the interconnection of \emph{fast decaying} plant dynamics and \emph{slow} optimization dynamics is asymptotically stable. The results in this section can be generalized to a dynamic plant accordingly.
	}
	The set $\calU$ defines constraints which are enforced by physical saturation.

	For solving~\eqref{eq:app_opt_prob} we aim at approximating the projected gradient flow
$\dot u = \Pi_\calU^G [- G^{-1}(u)\nabla \hat{\Phi} (u)](u)$,
	where we have defined $\hat{\Phi}(u) := \Phi (h(u))$ to eliminate the state variable~$x$. In particular, we have $\nabla \hat{\Phi}(u) = H^T \nabla \Phi(h(u))$. In the following, the metric $G$ will be either $G \equiv \bbI$ or $G \equiv Q$ (the latter yielding a projected Newton flow).

	To approximate~$\dot u = \Pi_\calU^G [-  G^{-1}(u) \nabla \hat{\Phi}(u) ](u)$, we consider three systems that fall into the class of anti-windup approximations defined by~\eqref{eq:sys_aw_gen}, two of which can be implemented in a feedback loop as in \cref{fig:aw1}. Their convergence behavior for the same problem instance and varying $K$ is illustrated in \cref{fig:gradient_sim} and discussed below.
	\begin{enumerate}[leftmargin=*, label=\emph{\roman*})]
		\item \emph{Penalty Gradient Flow:} As a reference system we consider the gradient flow of the potential function $\Psi(u) := \hat{\Phi}(u) + \tfrac{1}{2K} d^2_\calU(u))$ which is given by
		      \begin{align}\label{eq:sim_pen_grad}
			      \dot u = - \nabla \Psi(u) = -  H^T \nabla \Phi(h(u)) - \tfrac{1}{K}(u - P_\calU(u) ) \,.
		      \end{align}
		      In this case, we have $G \equiv \bbI$ and $K > 0$ takes the role of a \emph{penalty parameter} for the soft penalty term $d_\calU^2$ that approximately enforces the input constraint $u \in \calU$.\footnote{The penalty $d^2_\calU$ is illustrative in the context of autonomous optimization, however, it is not generally practical for numerical optimization, because evaluating $\nabla d^2_\calU$ requires computing $P_\calU$. Instead, in numerical applications, it is more common to use a penalty $\| \max \{ A_u u - b_u, 0 \}\|^2$.} The system~\eqref{eq:sim_pen_grad} is a special case of the AWA~\eqref{eq:sys_aw_gen} and, as a consequence, \cref{thm:loc,thm:rob_stab} (uniform convergence and robust practical stability) and their corollaries apply as $K \rightarrow 0^+$. However,~\eqref{eq:sim_pen_grad} is not of the special form~\eqref{eq:sys_aw_gen_spec} and convergence of to the optimizer of the problem~\eqref{eq:app_opt_prob} is not guaranteed for positive $K > 0$. Neither does~\eqref{eq:sim_pen_grad} lend itself to a feedback implementation, because $\nabla {\Phi}$ is evaluated at $h(u)$ rather than at $h(P_\calU(u))$ (which is the actual system state for the saturated input).

		\item \emph{Anti-Windup Gradient Scheme:} As a second type of dynamics we consider
		      \begin{align}\label{eq:sim_aw_grad}
			      \underbrace{\dot u = - H^T \nabla {\Phi}( \overline{x}) - \tfrac{1}{K} (u -\overline{u})}_{\text{controller}} \qquad \underbrace{\overline{u} := P_\calU(u) \qquad \overline{x} := h(\overline{u})}_{\text{physical system}}
		      \end{align}
		      which can be implemented in closed loop because the quantities  $\o{u}$ and $\o{x}$ are ``evaluated'' by the physical system at no computational cost (and are assumed to be measurable), which is one of the key features of autonomous optimization.

		      Furthermore, because $\calU$ is convex and $\Phi$ is strongly convex (which implies strong monotonicity), Theorem~\eqref{thm:monot} is applicable and guarantees that $\overline{z} = (\overline{u}, \overline{x})$ converges to the optimizer of~\eqref{eq:app_opt_prob}. This is confirmed in \cref{fig:gradient_sim}.

		\item \emph{Anti-Windup Newton Scheme:} As the final gradient-based anti-windup scheme we consider an anti-windup approximation with $G \equiv Q$ and which is given by
		      \begin{align}\label{eq:sim_aw_newt}
			      \underbrace{\dot u = - Q^{-1} \left(H^T \nabla {\Phi}( \overline{x}) - \tfrac{1}{K} (u -\overline{u} )\right)}_{\text{controller}} \qquad \underbrace{\overline{u} := P_\calU(u) \qquad \overline{x} := h(\overline{u})}_{\text{physical system}} \, .
		      \end{align}
		      The system~\eqref{eq:sim_aw_newt} can be implemented in closed loop with a physical system and approximates a \emph{projected Newton flow}~\cite[Ex.~5.6]{hauswirthProjectedDynamicalSystems2018}. This fact is noteworthy, because, in general, projected Newton flows do not lend themselves to an easy implementation (e.g., as an iterative algorithm).

		      Even though, as seen in \cref{fig:gradient_sim}, $\o{u}$ converges to the optimizer of~\eqref{eq:app_opt_prob}, strictly speaking, Theorem~\eqref{thm:monot} is not directly applicable because $Q \neq \bbI$.
	\end{enumerate}

	\begin{figure}[tb]
		\centering
		\includegraphics[width=.9\columnwidth]{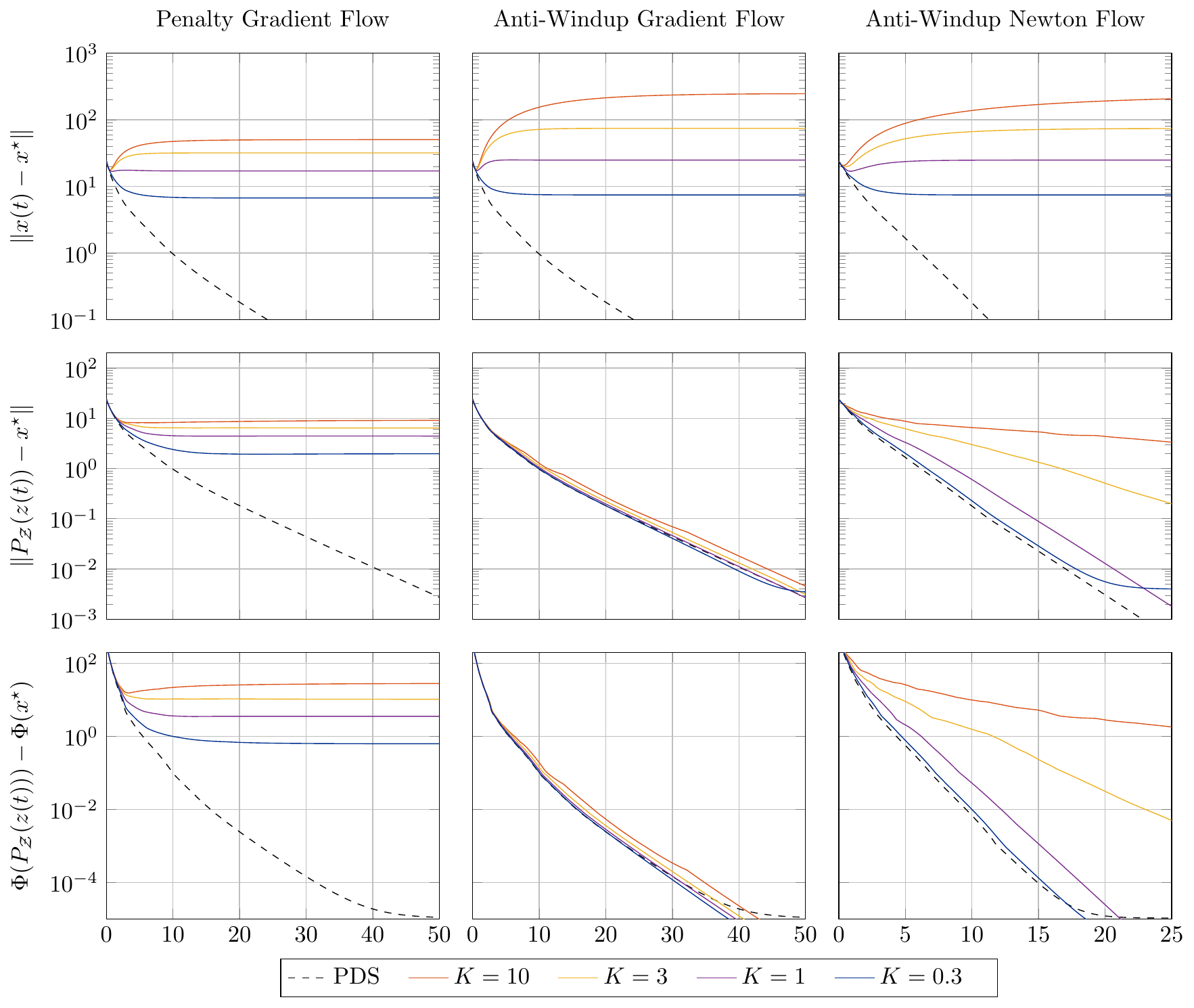}
		\caption[Simulations]{Convergence behavior of~\eqref{eq:sim_pen_grad},~\eqref{eq:sim_aw_grad}, and~\eqref{eq:sim_aw_newt} for a problem instance of~\eqref{eq:app_opt_prob} with $p = 100$ (input dimension) and $r = 300$ (\# of input constraints).\protect\footnotemark}
		\label{fig:gradient_sim}
	\end{figure}

	\footnotetext{The anti-windup dynamics are simulated with MATLAB using a fixed-stepsize forward Euler scheme. The projection on $\calU$ is evaluated using \texttt{quadprog}. The \emph{nominal} \pds is approximated using a projected forward Euler scheme as $u^{k+1} = P_\calU(u^k + \alpha f(u^k))$ which is guaranteed to converge uniformly as $\alpha \rightarrow 0^+$~\cite{nagurneyProjectedDynamicalSystems1996}.
	}

	The anti-windup gradient and Newton schemes defined above illustrate some of the key features of autonomous optimization and anti-windup implementations:
	\begin{enumerate}[leftmargin=*, label=\emph{\roman*})]
		\item Under the conditions of \cref{thm:monot}, the actual system state and saturated control input converge to the optimizer $u^\star$ of~\eqref{eq:app_opt_prob}, even though the internal control variable $u$ does not in general converge to $u^\star$.

		\item In a feedback implementation exploiting input saturation, neither the set $\calU$ nor the steady-state disturbance $w$ needs to be known (or estimated). The only model information required is $H$.
		      Furthemore, recent preliminary theoretical~\cite{colombinorobustnessguaranteesfeedbackbased2019} and experimental results for power systems~\cite{ortmannExperimentalValidationFeedback2019} suggest that these feedback schemes are robust against uncertainties in  $H$.

		\item The simulations in \cref{fig:gradient_sim} suggest that the convergence rate of the ``projected trajectory'' of~\eqref{eq:sim_aw_grad} is not affected by the value of $K$ and is equivalent to the convergence rate of the nominal projected gradient flow. In contrast, the convergence rate of the anti-windup Newton scheme~\eqref{eq:sim_aw_newt} does depend on $K$ and one can recover the rate of projected Newton flow only in the limit $K \rightarrow 0^+$. An analysis of this observation remains, however, outside the scope of this paper.
	\end{enumerate}

	\subsection{Feedback-based Saddle-Flows with Anti-Windup}\label{subsec:saddle}
	In autonomous optimization, constraints on the system state (or output) cannot be enforced directly because they are not directly controllable and often subject to disturbances affecting the physical plant (e.g. an unknown value of $w$). For the purpose of enforcing state or output constraints, projected saddle-point flows have been proven effective~\cite{tangRunningPrimalDualGradient2018,dallaneseOptimalPowerFlow2018,ortmannExperimentalValidationFeedback2019}. In this section, we indicate how anti-windup approximations can be combined with this type of dynamical system, even though this leads us slightly outside the scope of our theoretical results.
	We consider quadratic program
	\begin{align}\label{eq:app_opt_prob_out}
		\begin{split}
			\text{minimize} \quad &\Phi(x) \\
			\text{subject to} \quad & x = h(u), \, u \in \calU  \\
			& x \in \calX := \{ x \, | \, A_x x \leq b_x \} \, ,
		\end{split}
	\end{align}
	where $\Phi, h$, and $\calU$ are defined as in~\eqref{eq:app_opt_prob} and $\calX$ denotes a set of state constraints with $A_x \in \bbR^{s \times m}$ and $b_x \in \bbR^s$.
	To solve~\eqref{eq:app_opt_prob_out}, we consider the projected saddle-point flow
	\begin{align}\label{eq:saddle_flow_proj}
		\dot u = \Pi_\calU \left[- H^T \nabla \Phi(h(u)) - H^T A_x^T \mu \right] \qquad \qquad
		\dot \mu = \Pi_{\bbR^s_{\geq 0}} [ A_x h(u) - b_x] \, ,
	\end{align}
	where $\mu \in \bbR^s$ denotes the dual multipliers associated with the output constraints. The system~\eqref{eq:saddle_flow_proj} (and special cases in which either primal or dual variables are not projected) has been extensively studied and convergence is guaranteed, for instance, under strict convexity of $\Phi$. We refer the reader to~\cite{goebelStabilityRobustnessSaddlepoint2017,cherukuriRoleConvexitySaddlePoint2017} and references therein.

	We approximate~\eqref{eq:saddle_flow_proj} with a (partial) anti-windup implementation as
	\begin{align}\label{eq:saddle_flow_aw}
		\underbrace{
			\begin{aligned}
				\dot u   & = - H^T \nabla {\Phi}( \overline{x})  - H^T A_x^T \mu - \tfrac{1}{K} (u -\overline{u}) \\
				\dot \mu & = \Pi_{\bbR^s_\geq 0} [ A_x \o{x} - b_x]
			\end{aligned}
		}_{\text{controller}}
		\qquad \underbrace{\overline{u} := P_\calU(u) \qquad \overline{x} := h(\overline{u})}_{\text{physical system}} \, .
	\end{align}

	We do not approximate the projected integration of the dual variables with an anti-windup term, since the dual variables are often internal variables of the controller and the projection on the non-negative orthant is easily implementable.

	\begin{figure}
		\centering
		\includegraphics[width=.9\columnwidth]{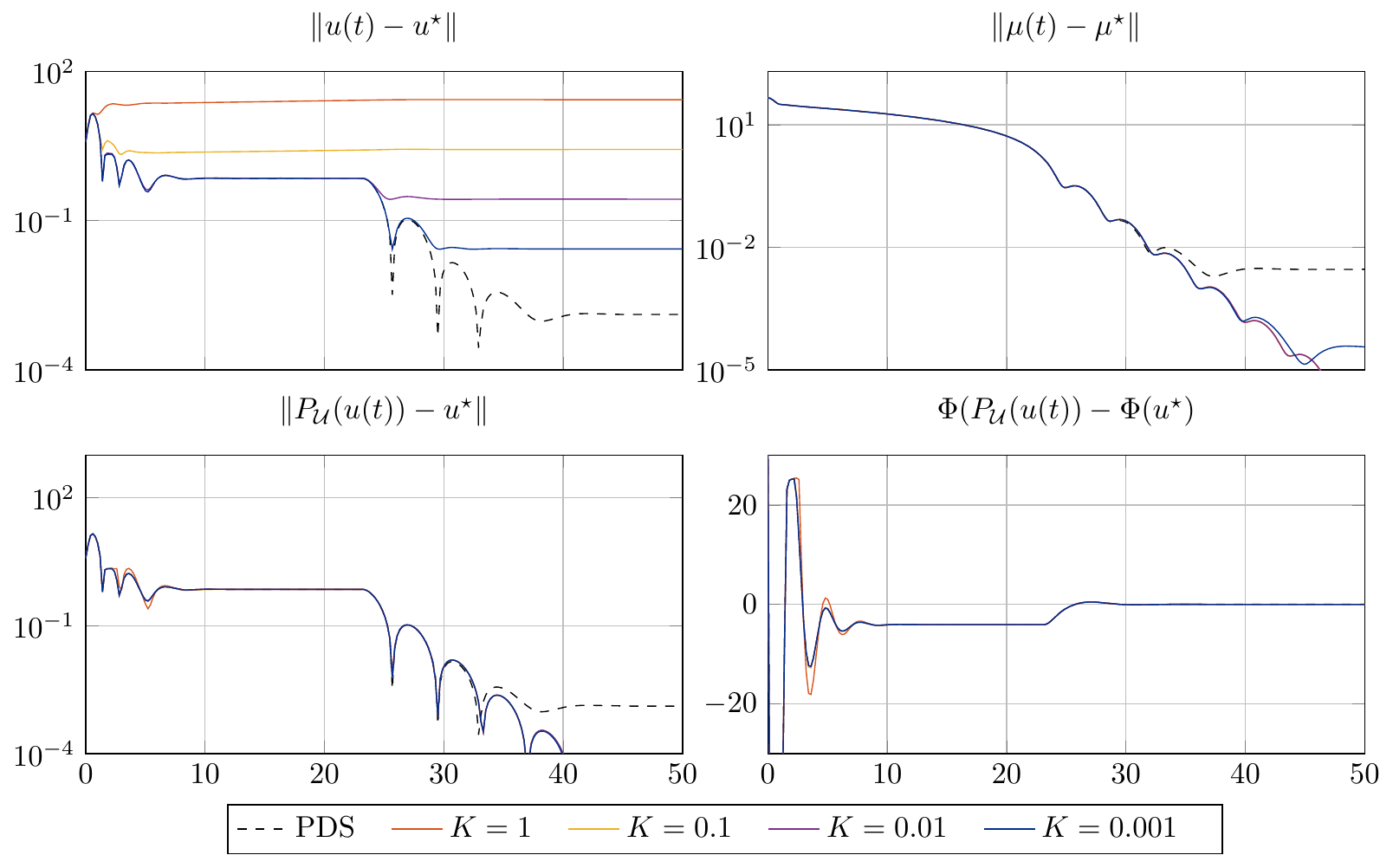}
		\caption[Simulations]{Convergence behavior of~\eqref{eq:saddle_flow_aw} (and the \pds~\eqref{eq:saddle_flow_proj}) for a problem instance of~\eqref{eq:app_opt_prob_out} with $p = 3$ (input dimension), $m = 5$ (state dimension), $r = 10$ (\# of input constraints), and $s = 5$ (\# of state constraints).}
		\label{fig:saddle_sim}
	\end{figure}

	\cref{fig:saddle_sim} illustrates the behavior of~\eqref{eq:saddle_flow_proj} and~\eqref{eq:saddle_flow_aw}. Similarly to the results for the gradient anti-windup approximations, we observe that $u$ does not, in general, converge to its optimal value. However, the saturated control input $P_\calU(u)$ (and thereby the actual system state) and the dual variable $\mu$ converge to the solution of~\eqref{eq:app_opt_prob_out}.

	\cref{thm:monot} (robust convergence) does not apply to~\eqref{eq:saddle_flow_aw}. First, while the projected saddle-flow~\eqref{eq:saddle_flow_proj} is monotone, strong monotonicity is usually not guaranteed~\cite{goebelStabilityRobustnessSaddlepoint2017,cherukuriRoleConvexitySaddlePoint2017}. Second, by applying only a partial anti-windup approximation, the  vector field remains discontinuous because of the projection of $\mu$ on $\bbR^s_{\geq 0}$.

\section{Conclusion}\label{sec:conc}

In this paper we have studied a general class of dynamical systems which are inspired by classical anti-windup control schemes. We have rigourosly established that these systems approximate oblique projected dynamical systems in terms of uniform convergence and semiglobal practical robust stability. Furthermore, we have shown that for a special case, and under an additional monotonicity assumption, these anti-windup approximations exhibit robust convergence to the equilibria of the limiting projected dynamical system.
We have further illustrated several ways in which our results apply in the context of autonomous optimization. In particular, we have shown how physical saturation can be exploited to drive a plant to an optimal steady state without explicit knowledge of the physically-enforced input domain.

Several points remain open: First, it is unclear whether our analysis can be extended to consider control laws that incorporate a proportional control component. Second, the strong monotonicity requirement for robust convergence to equilibria of a projected dynamical systems can presumably be relaxed. Third, our simulations suggest that certain anti-windup gradient schemes retain the same convergence rate as the limiting projected gradient flow, independently of the anti-windup gain. Fully understanding this surprising phenomenon requires further work.

\nocite{*}
\bibliographystyle{siamplain}
\bibliography{bibliography_static}

\begin{thebibliography}{10}

\bibitem{adlyPreservationProxRegularitySets2016}
{\sc S.~Adly, F.~Nacry, and L.~Thibault}, {\em Preservation of
  {{Prox}}-{{Regularity}} of {{Sets}} with {{Applications}} to {{Constrained
  Optimization}}}, SIAM J. Optim., 26 (2016), pp.~448--473.

\bibitem{aubinViabilityTheory1991}
{\sc J.~P. Aubin}, {\em Viability Theory}, Systems \& {{Control}}:
  {{Foundations}} \& {{Applications}}, {Springer}, {Boston}, 1991.

\bibitem{aubinDifferentialInclusionsSetValued1984}
{\sc J.-P. Aubin and A.~Cellina}, {\em Differential {{Inclusions}}:
  {{Set}}-{{Valued Maps}} and {{Viability Theory}}}, Grundlehren Der
  Mathematischen {{Wissenschaften}}, {Springer}, {Berlin Heidelberg}, 1984.

\bibitem{brogliatoEquivalenceComplementaritySystems2006}
{\sc B.~Brogliato, A.~Daniilidis, C.~Lemar{\'e}chal, and V.~Acary}, {\em On the
  equivalence between complementarity systems, projected systems and
  differential inclusions}, Syst Control Lett, 55 (2006), pp.~45--51.

\bibitem{cherukuriRoleConvexitySaddlePoint2017}
{\sc A.~Cherukuri, E.~Mallada, S.~Low, and J.~Cort{\'e}s}, {\em The {{Role}} of
  {{Convexity}} on {{Saddle}}-{{Point Dynamics}}: {{Lyapunov Function}} and
  {{Robustness}}}, IEEE Trans. Autom. Control, 63 (2017), pp.~2449--2464.

\bibitem{colombinoOnlineOptimizationFeedback2019}
{\sc M.~Colombino, E.~Dall'Anese, and A.~Bernstein}, {\em Online
  {{Optimization}} as a {{Feedback Controller}}: {{Stability}} and
  {{Tracking}}}, IEEE Trans. Control Netw. Syst.,  (2019).

\bibitem{colombinorobustnessguaranteesfeedbackbased2019}
{\sc M.~Colombino, J.~W. {Simpson-Porco}, and A.~Bernstein}, {\em Towards
  robustness guarantees for feedback-based optimization}, ArXiv190507363 Math,
  (2019).

\bibitem{cornetExistenceslowsolutions1983}
{\sc B.~Cornet}, {\em Existence of slow solutions for a class of differential
  inclusions}, Journal of Mathematical Analysis and Applications, 96 (1983),
  pp.~130--147.

\bibitem{cortesDiscontinuousdynamicalsystems2008}
{\sc J.~Cort{\'e}s}, {\em Discontinuous dynamical systems}, IEEE Control Syst.
  Mag., 28 (2008), pp.~36--73.

\bibitem{dallaneseOptimalPowerFlow2018}
{\sc E.~Dall'Anese and A.~Simonetto}, {\em Optimal {{Power Flow Pursuit}}},
  IEEE Trans. Smart Grid, 9 (2018), pp.~942--952.

\bibitem{depersisDistributedaveragingintegral2019}
{\sc C.~De~Persis and S.~Grammatico}, {\em Distributed averaging integral
  {{Nash}} equilibrium seeking on networks}, Automatica, 110 (2019), p.~108548.

\bibitem{facchineiFinitedimensionalvariationalinequalities2003}
{\sc F.~Facchinei and J.-S. Pang}, {\em Finite-Dimensional Variational
  Inequalities and Complementarity Problems}, Springer Series in Operations
  Research, {Springer}, {New York}, 2003.

\bibitem{filippovDifferentialEquationsDiscontinuous1988}
{\sc A.~F. Filippov}, {\em Differential {{Equations}} with {{Discontinuous
  Righthand Sides}}: {{Control Systems}}}, Mathematics and Its {{Applications}}
  ({{Soviet Series}}), {Springer Netherlands}, 1988.

\bibitem{goebelStabilityRobustnessSaddlepoint2017}
{\sc R.~Goebel}, {\em Stability and robustness for saddle-point dynamics
  through monotone mappings}, Syst Control Lett, 108 (2017), pp.~16--22.

\bibitem{goebelHybridDynamicalSystems2012}
{\sc R.~Goebel, R.~G. Sanfelice, and A.~R. Teel}, {\em Hybrid {{Dynamical
  Systems}}: {{Modeling}}, {{Stability}}, and {{Robustness}}}, {PUP}, 2012.

\bibitem{hauswirthProjectedDynamicalSystems2018}
{\sc A.~Hauswirth, S.~Bolognani, and F.~D{\"o}rfler}, {\em Projected
  {{Dynamical Systems}} on {{Irregular}}, {{Non}}-{{Euclidean Domains}} for
  {{Nonlinear Optimization}}}, ArXiv180904831 Math,  (2018).

\bibitem{hauswirthProjectedGradientDescent2016}
{\sc A.~Hauswirth, S.~Bolognani, G.~Hug, and F.~D{\"o}rfler}, {\em Projected
  gradient descent on {{Riemannian}} manifolds with applications to online
  power system optimization}, in 54th {{Annual Allerton Conference}} on
  {{Communication}}, {{Control}}, and {{Computing}}, {Monticello, IL}, Sept.
  2016, pp.~225--232.

\bibitem{hauswirthTimescaleSeparationAutonomous2019}
{\sc A.~Hauswirth, S.~Bolognani, G.~Hug, and F.~D{\"o}rfler}, {\em Timescale
  {{Separation}} in {{Autonomous Optimization}}}, ArXiv190506291 Math,  (2019).

\bibitem{hauswirthImplementationProjectedDynamical2020}
{\sc A.~Hauswirth, F.~D{\"o}rfler, and A.~Teel}, {\em On the {{Implementation}}
  of {{Projected Dynamical Systems}} with {{Anti}}-{{Windup Controllers}}}, in
  American {{Control Conference}} ({{ACC}}), 2020, {Denver, CO}, July 2020.
\newblock accepted.

\bibitem{hauswirthOnlineOptimizationClosed2017}
{\sc A.~Hauswirth, A.~Zanardi, S.~Bolognani, F.~D{\"o}rfler, and G.~Hug}, {\em
  Online optimization in closed loop on the power flow manifold}, in 2017
  {{IEEE PowerTech}}, {Manchester, UK}, June 2017.

\bibitem{hiriart-urrutyFundamentalsConvexAnalysis2012}
{\sc J.-B. {Hiriart-Urruty} and C.~Lemar{\'e}chal}, {\em Fundamentals of
  {{Convex Analysis}}}, Grundlehren {{Text Editions}}, {Springer}, {Berlin
  Heidelberg}, 2012.

\bibitem{kellyRatecontrolcommunication1998}
{\sc F.~P. Kelly, A.~K. Maulloo, and D.~K.~H. Tan}, {\em Rate control for
  communication networks: Shadow prices, proportional fairness and stability},
  J Oper Res Soc, 49 (1998), pp.~237--252.

\bibitem{lahkarProjectionDynamicGeometry2008}
{\sc R.~Lahkar and W.~H. Sandholm}, {\em The projection dynamic and the
  geometry of population games}, Games and Economic Behavior, 64 (2008),
  pp.~565--590.

\bibitem{lawrenceOptimalSteadyStateControl2018}
{\sc L.~S.~P. Lawrence, Z.~E. Nelson, E.~Mallada, and J.~W. {Simpson-Porco}},
  {\em Optimal {{Steady}}-{{State Control}} for {{Linear Time}}-{{Invariant
  Systems}}}, in 2018 {{IEEE Conference}} on {{Decision}} and {{Control}}
  ({{CDC}}), {Miami Beach, FL}, Dec. 2018, pp.~3251--3257.

\bibitem{lowInternetCongestionControl2002}
{\sc S.~H. Low, F.~Paganini, and J.~C. Doyle}, {\em Internet congestion
  control}, IEEE Control Syst. Mag., 22 (2002), pp.~28--43.

\bibitem{mentaStabilityDynamicFeedback2018}
{\sc S.~Menta, A.~Hauswirth, S.~Bolognani, G.~Hug, and F.~D{\"o}rfler}, {\em
  Stability of dynamic feedback optimization with applications to power
  systems}, in 56th {{Annual Allerton Conference}} on {{Communication}},
  {{Control}}, and {{Computing}}, {Monticello, IL}, Oct. 2018, pp.~136--143.

\bibitem{molzahnSurveyDistributedOptimization2017}
{\sc D.~K. Molzahn, F.~D{\"o}rfler, H.~Sandberg, S.~H. Low, S.~Chakrabarti,
  R.~Baldick, and J.~Lavaei}, {\em A {{Survey}} of {{Distributed Optimization}}
  and {{Control Algorithms}} for {{Electric Power Systems}}}, IEEE Trans. Smart
  Grid, 8 (2017), pp.~2941--2962.

\bibitem{nagurneyProjectedDynamicalSystems1996}
{\sc A.~Nagurney and D.~Zhang}, {\em Projected {{Dynamical Systems}} and
  {{Variational Inequalities}} with {{Applications}}}, {Springer}, 1~ed., 1996.

\bibitem{ortmannExperimentalValidationFeedback2019}
{\sc L.~Ortmann, A.~Hauswirth, I.~Caduff, F.~D{\"o}rfler, and S.~Bolognani},
  {\em Experimental validation of feedback optimization in power distribution
  grids}, ArXiv191003384 Eess,  (2019).

\bibitem{rockafellarVariationalAnalysis2009}
{\sc R.~T. Rockafellar and R.~J.-B. Wets}, {\em Variational Analysis}, no.~317
  in Grundlehren Der Mathematischen {{Wissenschaften}}, {Springer},
  {Heidelberg}, 3~ed., 2009.

\bibitem{sontagCommentsIntegralVariants1998}
{\sc E.~D. Sontag}, {\em Comments on integral variants of {{ISS}}}, Syst
  Control Lett, 34 (1998), pp.~93--100.

\bibitem{tangRunningPrimalDualGradient2018}
{\sc Y.~Tang, E.~Dall'Anese, A.~Bernstein, and S.~Low}, {\em Running
  {{Primal}}-{{Dual Gradient Method}} for {{Time}}-{{Varying Nonconvex
  Problems}}}, ArXiv181200613 Math,  (2018).

\bibitem{tangRealTimeOptimalPower2017}
{\sc Y.~Tang, K.~Dvijotham, and S.~Low}, {\em Real-{{Time Optimal Power
  Flow}}}, IEEE Trans. Smart Grid, 8 (2017), pp.~2963--2973.

\bibitem{tarbouriechAntiwindupDesignOverview2009}
{\sc S.~Tarbouriech and M.~Turner}, {\em Anti-windup design: An overview of
  some recent advances and open problems}, IET Control Theory Appl., 3 (2009),
  pp.~1--19.

\bibitem{zaccarianModernAntiwindupSynthesis2011}
{\sc L.~Zaccarian and A.~R. Teel}, {\em Modern Anti-Windup Synthesis: Control
  Augmentation for Actuator Saturation}, {Princeton University Press}, 2011.

\end{thebibliography}
\end{document}